\documentclass[11pt,a4paper]{article}
\usepackage{a4wide,amsfonts,amsmath,latexsym,amsthm,amssymb,euscript,eufrak,graphicx,units,mathrsfs,setspace,enumerate}

\usepackage[pdftex]{color}
\definecolor{darkblue}{rgb}{0.3,0.3,0.7}
\definecolor{bucheron}{RGB}{250,100,24}

\usepackage{float}
\usepackage{caption}
\newfloat{figure}{H}{lof}
\floatname{figure}{\figurename}
\usepackage[french,english]{babel}
\usepackage[T1]{fontenc}
\usepackage[utf8]{inputenc}
\usepackage{graphicx}
\usepackage{upgreek}
\usepackage{marginnote}
\usepackage{tikz}

\DeclareMathAlphabet{\eufrak}{U}{}{}{}  
\SetMathAlphabet\eufrak{normal}{U}{euf}{m}{n}
\SetMathAlphabet\eufrak{bold}{U}{euf}{b}{n}

\newtheorem{prop}{Proposition}[section]

\newtheorem{theorem}[prop]{Theorem}
\newtheorem{lemma}[prop]{Lemma}
\newtheorem{corollary}[prop]{Corollary}
\newtheorem{proposition}[prop]{Proposition}

\theoremstyle{definition}

\newtheorem{remark}[prop]{Remark}

\newtheorem{assumption}[prop]{Assumption}

\newcounter{Fact} 

\numberwithin{equation}{section}

\newcommand{\pp}{\leqslant}      
\renewcommand{\textless}{\langle}  
\renewcommand{\textgreater}{\rangle}  
\newcommand{\id}[2]{[\![#1,#2]\!]} 
\newcommand{\f}[1]{\EuScript{F}_{#1}}   
\newcommand{\g}[1]{\EuScript{G}_{#1}} 
\newcommand{\Bi}{\mathfrak b} 
\newcommand{\cA}{\mathcal A} 
\newcommand{\cB}{\mathcal B} 
\newcommand{\sG}{\mathscr G} 
\newcommand{\sP}{\mathscr P} 
\newcommand{\sS}{\mathscr S} 

\newcommand{\D}{\Delta} 
\newcommand{\F}{\EuScript{F}} 
\newcommand{\G}{\EuScript{G}} 
\newcommand{\HH}{\EuScript{H}} 
\newcommand{\Cm}{\mathrm C} 
\newcommand{\cm}{\mathrm c} 
\newcommand{\Xm}{\mathrm X} 
\newcommand{\Dm}{\mathrm D} 
\newcommand{\Um}{\mathrm U} 
\newcommand{\Ym}{\mathrm Y} 
\newcommand{\Zm}{\mathrm Z} 
\newcommand{\Fm}{\mathrm F} 
\newcommand{\Gm}{\mathrm G} 
\newcommand{\Km}{\mathrm K} 
\newcommand{\Rm}{\mathrm R} 
\newcommand{\Nm}{\mathrm N} 
\newcommand{\Vm}{\mathrm V} 
\newcommand{\gP}{\mathbf P} 
\newcommand{\Q}{\mathbf Q} 
\newcommand{\M}{\mathbf M} 
\newcommand{\St}{\overline{\mathrm S}} 
\newcommand{\Vt}{\overline{\mathrm V}} 
\newcommand{\Sm}{\mathrm S} 
\newcommand{\q}{\mathbf q} 
\newcommand{\Xb}{\mathbb X} 
\newcommand{\Lm}{\mathrm L} 
\newcommand{\A}{\mathrm A} 
\newcommand{\B}{\mathrm B} 
\newcommand{\vp}{\varphi} 
\newcommand{\fk}{\mathfrak f} 
\newcommand{\car}{\mathbf 1} 
\newcommand{\gE}{\mathbf E} 
\newcommand{\esp}[1]{\mathbf E\left[#1\right]} 

\newcommand{\R}{\mathbb{R}}
\newcommand{\N}{\mathbb{N}}

\newcommand{\bM}{{\mathbf M}}

\newcommand{\gPs}{\widehat{\mathbf P}}
\newcommand{\gPb}{\mathbf P^{\mathfrak b}}
\newcommand{\gPc}{\mathbf P^{\mathfrak c}}

\newcommand{\wPb}{\widehat{\mathbf P}^{\Bi}}
\newcommand{\wP}{\widehat{\mathbf P}}

\newcommand{\wQb}{\widehat{\mathbf Q}^{\Bi}}
\newcommand{\wQ}{\widehat{\mathbf Q}}

\newcommand{\pow}{\mathbf{pow}}

\begin{document}

\title{The insider trading problem in a jump-binomial model.
}

\author{Hélène Halconruy \footnote{Léonard de Vinci Pôle universitaire, Research Center, Paris La Défense, France.} \footnote{Modal'X, Université Paris Nanterre, Nanterre, France. \textit{e-mail adress:} helene.halconruy@devinci.fr}}

\date{}

\maketitle


\begin{abstract} 
We study insider trading in a jump-binomial model of the financial market that is based on a marked binomial process and that serves as a suitable alternative to some classical trinomial models. 
Our investigations focus on the two main questions : measuring the advantage of the insider's additional information and stating a closed form for her hedging strategy. Our approach is based on the results of enlargement of filtration in a discrete setting stated by Blanchet-Scalliet and Jeanblanc \cite{Blanchet_Jeanblanc_2020} and on a stochastic analysis for marked binomial processes developed in the companion paper \cite{Halconruy_2021_binomial}. Our work provides in a discrete-time and an incomplete market setting the analogues of some results of Amendinger \textit{et al.} \cite{Amendinger_2000, Amendinger_Becherer_Schweizer_2003}, Imkeller \textit{et al.} \cite{Imkeller_1996,Imkeller_2003} and extends in an insider framework some utility maximization results stated in Delbaen and Schachermayer \cite{Delbaen_Schachermayer} and in Runggaldier \cite{Runggaldier_2006}.
\vspace{5pt}\\

\noindent
\textbf{Keywords} Insider trading, trinomial model, enlargement of filtrations, Malliavin's calculus, utility maximization.\vspace{2pt}\\
\noindent
\textbf{PACS} 60J74, 60G55, 91G20, 60H30,60H07, 91G10, 94A17. \vspace{2pt}\\
\noindent
\textbf{JEL Classification} G11, G14, C61, C02.
\end{abstract}

\noindent
\textbf{Acknowledgements}
This project has received funding from the European Union’s Horizon 2020 research and innovation programme under grant agreement $\Nm^\circ 811017$. I warmly thank Giovanni Peccati for his recommendations and advice. I am very grateful to Monique Jeanblanc for her invaluable help and the motivating discussions while revising this article.

\tableofcontents

\section{Introduction}

This paper addresses some aspects of insider trading in a jump-binomial model of the financial market that is comparable to some trinomial models (defined in Boyle \textit{et al.} \cite{Boyle_1988,Boyle_Kirzner_1985}).
Throughout the trading period, the insider has access to hidden additional information encapsulated in a random variable G, the outcome of which she knows from the outset.
We approach the following questions from two different perspectives. First, we focus on the additional information itself and attempt to quantify the benefits it provides. Second, we adopt the insider's perspective to establish an \textit{optimal} hedging formula for certain replicable claims.\\

The model we consider involves two investors: an ordinary agent and an insider. Both are assumed to be small enough not to impact market prices, and the insider has exclusive, advantageous confidential information right from the start.
From the perspective of martingale theory adopted in this paper, the extra information is hidden in a random variable G, the outcome of which is known by the insider at the beginning of the trading interval. As a result, the insider's level of information is described by a filtration $\G$ that is larger than $\F$, which describes the ordinary agent’s level of information. This framework is naturally connected to the theory of \textit{enlargement of filtration}, which can be roughly classified into two distinct approaches: the \textit{initial enlargement} approach under Jacod's hypothesis, assuming equivalence between the conditional laws of G with respect to $\F$ and the law of G (see Jacod \cite{Jacod_1985}), and the \textit{progressive enlargement approach} (see Barlow \cite{Barlow_1978}, Jeulin and Yor \cite{Jeulin_Yor_1978}). All related results extend immediately to a discrete time setting, as highlighted by Blanchet-Scalliet, Jeanblanc in \cite{Blanchet_Jeanblanc_2020} and with Romero in \cite{Blanchet_Jeanblanc_Romero_2019}, most of them simply stemming from Doob’s decomposition. \\
The theory partly owes its success to applications in finance and notably to insider trading problems (see Kohatsu-Higa \cite{Kohatsu_2004}). One of the questions that arises is how to optimize the insider's expected utility and quantify her benefit. This is studied in Pikovsky and Karatzas \cite{Pikovsky_Karatzas}, Amendinger \textit{et al.} through  \cite{Amendinger_2000} and \cite{Amendinger_Becherer_Schweizer_2003}, as well as in Grorud and Pontier \cite{Grorud_Pontier_1998}. In \cite{Amendinger_Imkeller_Schweizer_1998} and \cite{Ankirchner_Dereich_Imkeller_2006}, Imkeller \textit{et al.} discover a crucial link between the insider's additional logarithmic utility and information theory by identifying it with the Shannon entropy of the extra information. In \cite{Imkeller_2003}, Imkeller connects these notions to Malliavin calculus by expressing the information drift as the logarithmic Malliavin trace of a conditional density characterizing the insider's advantage. Ankirchner \textit{et al.} describe in \cite{Ankirchner_Dereich_Imkeller_2006} the same information drift in a general setting and link it to the measure of the different levels of information contained in the agent and insider's filtrations. Finally, we would like to mention the comprehensive book (and the references therein) by Hillairet and Jiao \cite{Hillairet_Jiao}, which offers an exhaustive review of optimization results with exotic filtrations, especially in the context of insider trading.\\

In a related stream of research, insider trading appears as a byproduct of portfolio management issues (see Biagini and \O ksendal \cite{Biagini_Oksendal_2005,Biagini_Oksendal_2006}). An extensive literature on related topics (see Shreve \cite{Shreve_2005}, Pascucci and Runggaldier \cite{Pascucci_Runggaldier_2012}) is available in the most famous complete discrete-time market model, namely, the \textit{Cox-Ross-Rubinstein} or \textit{binomial} model. All claims are replicable in this model, and Privault provides an explicit formula of the hedging strategy in terms of the discrete Malliavin derivative for Rademacher processes (see  \cite{Privault_stochastica}, chapter 1 or \cite{Privault_2013}).\\
Even if the trinomial model is an interesting case study as the simplest incomplete market in discrete time, research about portfolio management in this frame is scarcer. We can nevertheless mention the books of and Schachermayer \cite{Delbaen_Schachermayer} or Bj\"orefeldt \textit{et al.} \cite{Bjorefeldt_Hee_Malmgard_Niklasson_Pettersson_Rado_2016}, the survey of Runggaldier \cite{Runggaldier_2006}, the work of Dai and Lyuu \cite{Dai_Lyuu_2010} and that of Glonti \textit{et al.} \cite{Glonti_Jamburia_Kapanadze_Khechinashvili_2002}. From a slightly different perspective, a hedger can aim at maximizing her expected utility from the terminal wealth for a given utility function. A very popular method is based on the formulation of a dual problem; the reader can refer to the survey of Schachermayer \cite{Schachermayer_2001} and the reference book of Delbaen and Schachermayer \cite{Delbaen_Schachermayer}. The same question of utility optimization has also been addressed in incomplete markets in a "classical" sense (Hu, Imkeller and Muller \cite{Hu_Imkeller_Muller_2005}) and, more recently, within other types of incompleteness such as that arising from friction (see Bouchard and Nutz \cite{Bouchard_Nutz_2015}, Neufeld and Sikic \cite{Neufeld_Sikic_2018}) or from uncertainty (see Nutz \cite{Nutz_2016}, Rasonyi and Meireles  \cite{Rasonyi_Meireles_2021}, Obl\'oj and Wiesel \cite{Obloj_Wiesel_2021}).
\vspace{3pt}\\

We present our results from two consecutive angles, each addressed in its own separate section: that of information theory, focusing on the additional knowledge possessed by the insider (Section 3) and that of the insider as a particular investor (Section 4).\\
Most of our more significant results are gathered in Section 3, where we focus on the additional information enjoyed by the insider. To gauge the advantages it confers, we compare the expected (logarithmic, exponential, power) utilities of both the ordinary agent and the insider agent. We measure the insider's benefit in the form of additional utility and link it to the entropy of the random variable G. These latter results accommodate works by Amendinger \textit{et al.} \cite{Amendinger_Becherer_Schweizer_2003}, \cite{Amendinger_2000}, Imkeller \textit{et al.} \cite{Amendinger_Imkeller_Schweizer_1998} to the incomplete discrete-time setting \cite{Ankirchner_Dereich_Imkeller_2006}. Our findings about utility optimization extend to an insider paradigm that of Delbaen and Schachermayer \cite{Delbaen_Schachermayer} and of Runggaldier \cite{Runggaldier_2006} holding in a classical trinomial model.\\
In Section 4, we delve into the insider's viewpoint to provide two complementary results. Thus, we propose a new interpretation of the \textit{information drift} that governs martingale preservation when shifting from the agent's level of information to the insider's enriched information level. In addition, we provide an explicit expression for the optimal hedging strategies for replicable claims with respect to the set of \textit{optimal martingale measures} identified in Section 3. Both results are derived not only from enlargement filtration theory but also from the Malliavin calculus for marked binomial processes developed in the companion paper \cite{Halconruy_2021_binomial}. Our Ocone-Karatzas-type formula for replicable claims extends Proposition 1.14.4 in Privault \cite{Privault_stochastica} to a discrete incomplete market model, while the results connected to enlargement of filtration illustrate that of Blanchet-Scalliet and Jeanblanc \cite{Blanchet_Jeanblanc_2020} in a simple incomplete market model.
\vspace{3pt}\\

The approach of the paper is original in two respects. First, and to our knowledge, the insider trading problem has thus far not been investigated in a discrete-time setting and in an incomplete market model. The simplest of them, the trinomial market model, is in this perspective an excellent case study: it enables comparison of the results with the continuous case (in particular to the Black-Scholes model to which it converges) or with the complete discrete-time binomial model.\\
Moreover, we give a new and useful representation of the classical trinomial market model, viewed here as a volatility model. To that end, we introduce the \textit{jump-binomial model} by replacing the sequence of i.i.d. random variables in $\{-1,0,1\}$ that underlie the trinomial market model by a discrete-time jump process called the \textit{marked binomial process}.
The advantage of working in this surrogate model is twofold. This is practically advantageous: its volatility structure allows for reasoning, conditioned on jump occurrences, in the binomial model where computations can easily be carried out. Furthermore, this enables us to harness the power of the Malliavin calculus framework for marked binomial processes developed in \cite{Halconruy_2021_binomial} for the first time for exclusively financial purposes. \vspace{3pt}\\

\noindent
The paper is organized as follows. In Section 2, we introduce the necessary instruments, including the jump-binomial model, tools of stochastic analysis for marked binomial point processes developed in \cite{Halconruy_2021_binomial}, and the results on enlargement of filtration from Blanchet-Scalliet and Jeanblanc \cite{Blanchet_Jeanblanc_2020}. The main findings of the paper are discussed in sections 3 and 4. In Section 3, we focus on the benefits provided by the additional information. We compute the expected utilities for both the agent and the insider, and compare them through the computation of the additional expected utility. We link this latter to information theory via the Shannon entropy of the random variable G. In Section 4, adopting the insider's point of view, we compute her hedging strategy. The main results, along with some perspectives, are summarized in the conclusion, Section 5. All proofs are postponed to the Appendix.

\section{Preliminaries and theoretical tools}\label{Section_preliminaries}
\subsection{The jump-binomial model: frame and martingale measures}\label{ternary_model_subsec}
\noindent
Throughout, we denote $\N_0=\N\cup\{0\}$ and we write $\id{n}{m}=\{n,n+1,\dots,m\}$ for any $n,m\in\N_0$ such that $n< m$. For $T\in\N$, let us define  $\Xb:=\id{1}{T}\times\{-1,1\}$ and $\mathcal X:=\sigma\{(t,k),\,t\in\id{1}{T},\,k=\pm1\}$.\\
Let $(\Omega,\cA,\gP)$ be an abstract probability space supposed to be rich enough to contain all random elements that must be defined.

\noindent
\underline{\textbf{The marked binomial process (MBP)}} on
 $\Xb$, denoted by $\eta$, can be constructed and defined as follows:
\begin{enumerate}
\item Consider $T$ independent Bernoulli experiments where a \textit{success} stands for a jump and occurs with probability $\lambda$. The random variable $\Nm_t\sim\mathcal B\mathrm{in}(t,\lambda)$ counts the number of jumps until time $t$.
\item If there is a jump at time $t$, draw a mark $k\in\{-1,1\}$ according to a probability distribution $\mathbf V$ on $\{-1,1\}$ and let $\eta(t,k)=1$ and $\eta(t,\cdot):=\eta(t,1)+\eta(t,-1)=1$. Otherwise, if there is no jump at time $t$, let $\eta(t,1)=\eta(t,-1)=0$ so that $\eta(t,\cdot)=0$.
\end{enumerate}
Then, the random variables $\Delta\Nm_t:=\Nm_{t}-\Nm_{t-1}=\eta(t,1)+\eta(t,-1)$ are independent Bernoulli random variables, and for any $k,\ell\in\{-1,1\}$, $t,s\in\id{1}{T}$ such that $t\neq s$, $\eta(t,k)$ and $\eta(s,\ell)$ are independent. \vspace{3pt}\\
This means that $\eta$ can be identified to the set of elements of $\Xb$ it \textit{lights up} as illustrated below.
\begin{figure}
\caption{\textbf{Realization of a MBP on $\id{1}{6}\times\{-1,1\}$}}
\begin{tikzpicture}
\draw[step=0.5cm,color=black!70] (0.5,0.5) grid (4,2.5);
\draw[->] (0.5,0.5) -- (0.5,3)node[left,scale=0.9]{$k$} ;
\draw[->] (0.5,1.5) -- (4.5,1.5)node[below,scale=0.9]{$t$} ;
\node(00) at (0.3,1.5){$0$};
\node(10) at (0.2,2){$1$};
\node(01) at (0.2,1){$-1$};
\fill[color=orange] (1.5,2) circle (3pt);
\fill[color=orange] (2,1) circle (3pt);
\fill[color=orange] (3.5,2) circle (3pt);
\fill[color=blue] (3,2) circle (3pt);
\draw[color=orange,fill=white] (1,1.5) circle (3pt);
\draw[color=orange,fill=white] (2.5,1.5) circle (3pt);
\node(eta) at (8.2,2.5){$\textcolor{orange}{\eta}$\,is\,defined\,on\,$(\id{1}{6}\setminus\{5\})\times\{-1,1\}$.};
\node(eta) at (7.2,2){ \underline{For this realization of $\eta$}:};
\node(eta) at (10.3,1.5){$\textcolor{orange}{\eta}(2,1)=\textcolor{orange}{\eta}(3,-1)=\textcolor{orange}{\eta}(6,1)=1$ : $(2,1),(3,-1),(6,1)$ are ON};
\node(noteta) at (8.6,1){$\textcolor{orange}{\eta}(1,\cdot)=\textcolor{orange}{\eta}(4,\cdot)=0$ : $(1,\cdot),(4,\cdot)$ are OFF};
\node(eta) at (8.3,0.5){$\textcolor{orange}{\eta}+\textcolor{blue}{\delta_{(5,1)}}$\,is\,defined\,on\,$\id{1}{6}\times\{-1,1\}$.};
\end{tikzpicture}
\end{figure}
\noindent
\underline{\textbf{The probability space}}
We may (and will) assume that  $\mathcal A=\F_T$ where $\f{}:=(\f{t})_{t\in\id{0}{T}}$ is the canonical filtration defined from $\eta$ by
\begin{equation*}
\f{0}:=\{\emptyset,\Omega\} \quad \text{and} \quad \f{t}:=\sigma\big(\eta(s,k),\,s\leqslant t,k=\pm 1\big)
\vspace{3pt}\\
\end{equation*}
The \textit{intensity} of $\eta$  is the measure $\nu$ on $\Xb$ defined for any $\A\in\mathcal X$, 
\begin{equation}\label{Eq_intensity_def}
\nu(\A)=\sum_{(t,k)\in\A}v(t,k)\delta_{(t,k)} \quad\text{with}\quad v(t,\pm 1):=\lambda \mathbf V(\{\pm 1\})=:\lambda p_{\pm 1}.
\end{equation}
In particular we have
\begin{equation}\label{Eq_lambda_p_def}
\lambda=\gP(\{\eta(1,\cdot)=1\}) \quad\text{and}\quad p:=p_{1}=1-p_{-1}.\\
\end{equation}
\noindent
\underline{\textbf{The \textit{jump-binomial model}}} defined on $(\Omega,\cA,(\F_t)_{t\in\id{0}{T}},\gP)$ embodies a simple financial market modelled by two assets, i.e., a couple of $\R_+$-valued processes $(\A_t,\Sm_t)_{t\in\id{0}{T}}$, defined on the same filtered probability space where
$\id{0}{T}$ is the \textit{trading interval} and  $T\in\N$ is the \textit{maturity}. 
The \textit{riskless asset} $(\A_t)_{t\in\id{0}{T}}$ is deterministic and is defined for some $r\in\R_+$ ($r$ is generally smaller than $1$) and for all $t\in\id{0}{T}$,
\begin{equation}\label{risklessdisc}
\A_t=(1+r)^t.
\end{equation}
The stock price which models the \textit{risky asset}, is the $\F$-adapted process $(\Sm_t)_{t\in\id{0}{T}}$ with (deterministic) initial value $\Sm_0=1$ that satisfies for any $t\in\id{1}{T}$,
\begin{equation}\label{riskternary}
\Delta \Sm_t:=\Sm_t-\Sm_{t-1}:=\theta_t\,\Sm_{t-1},
\end{equation}
where $\theta_t=r\car_{\{\eta(t,\cdot)=0\}}+b\car_{\{\eta(t,1)=1\}}+a\car_{\{\eta(t,-1)=1\}}$ and $a,b$ are real numbers such that $-1<a<0\leqslant r<b$. The sequence of discounted prices $\St:=(\St_t)_{t\in\id{0}{T}}$ is defined by $\St_t=\A_t^{-1}\Sm_t$ ($t\in\id{0}{T}$). 
Let us remark that \textbf{the $\theta_t$ are independent} as a consequence of the independence of variables  $\eta(t,k)$ and $\eta(s,\ell)$ for $t\neq s$ and $k,\ell\in\{-1,1\}$.
\vspace{3pt}\\
\noindent
The jump-binomial model holds significant practical interest, as emphasized by the two following remarks. As noted in Remark 2.1, it can be interpreted as a discrete stochastic volatility model, while Remark 2.2, reveals a \textit{correspondence} between this model and some trinomial models. It appears then to be a more suitable alternative to our problem and all the results are possible by virtue of this correspondence. 

\begin{remark}
The parameter $\lambda=\gP(\{\eta(1,\cdot)=1\})\in(0,1)$ can be viewed as the \textit{volatility} of the model: The closer $\lambda$ is to $0$, the lower the probability that the stock price process changes between the times $t-1$ and $t$, and the lower the volatility. Conversely, when $\lambda$ is close to $1$, there is a high probability of changes in the stock market process between $t-1$ and $t$. In the extreme case where $\lambda=1$, the surrogate model no longer corresponds to the trinomial model but coincides with the Cox-Ross-Rubinstein (or binomial) model. Let $\gPb$ be the probability measure on $(\Omega,\cA)$ defined by
\begin{equation*}
\gPb(\{\eta(1,\cdot)=1\})=1\quad\text{and}\quad \gPb(\{\eta(1,1)=1\})=p,
\end{equation*}
i.e., under which the probability of an occurrence of a jump at each time is $1$.  
We can remark the process $(\Sm_t/\Sm_{t-1})_{t\in\id{1}{T}}$ behaves under $\gPb$ as a \textit{binomial} or \textit{Rademacher} process. \textbf{In the sequel, we will refer $(\Omega,\cA,(\F_t)_{t\in\id{0}{T}},\gPb,\Sm)$ as the \textit{binomial model}.} 
\\
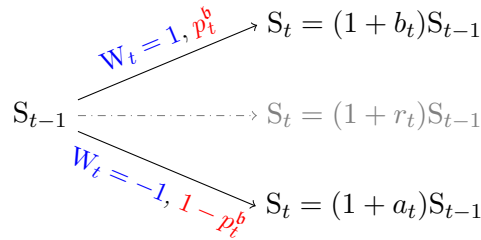
\begin{center}
\begin{tikzpicture}[scale=0.8]
%
%
 \node[] (RA) at (6,0){\textbf{Binomial model} \; $(\Omega,\cA,(\F_t)_{t\in\id{0}{T}},\gPb,\Sm)$};
\node[scale=1] (RA0) at (3,-3.5){$\Sm_{t-1}$};
\node[scale=1] (RA1U) at (8.5,-2){$\Sm_{t}=(1+b)\Sm_{t-1}$};
\draw[->](RA0) -- (RA1U.west) node[sloped,midway,above,scale=0.9]{\textcolor{blue}{$\eta(t,1)=1$},$\;\textcolor{red}{p}$};
%
\node[scale=1] (RA1D) at (8.5,-5){$\Sm_{t}=(1+a)\Sm_{t-1}$};
\draw[->] (RA0) -- (RA1D.west) node[sloped,midway,below,scale=0.9]{\textcolor{blue}{$\eta(t,-1)=1$},$\;\textcolor{red}{1-p}$};

\end{tikzpicture}
\end{center}
\end{remark}
\noindent
\begin{remark}
The jump-binomial model thus introduced is in fact a surrogate to a classical \textit{trinomial model} (see for instance Runggaldier \cite{Runggaldier_2006}, section 3.2.1) where $m$ would be equal to $1+r$. 
Let us recall - in a different probability space $(\Omega^{\mathrm{tri}},\cA^{\mathrm{tri}},\gP^{^{\mathrm{tri}}})$ - the definition of the classical trinomial model (where here $1+b,1+a,1+r$ stand respectively for the \textit{up}, \textit{down} and \textit{middle} parameters). The price process $(\Sm_t^{\mathrm{tri}})_{t\in\id{0}{T}}$ is characterized by its initial value $\Sm_0^{\mathrm{tri}}$ and satisfies for all $t\in\id{1}{T}$,
\begin{equation*}
\Delta\Sm_t^{\mathrm{tri}}=\Big[b\car_{\{\Xm_t^{\mathrm{tri}}=1\}}+a\car_{\{\Xm_t^{\mathrm{tri}}=-1\}}+r\car_{\{\Xm_t^{\mathrm{tri}}=0\}}\Big]\,\Sm_{t-1}^{\mathrm{tri}},
\end{equation*}
where $(\Xm_t^{\mathrm{tri}})_{t\in\id{1}{T}}$ is an i.i.d. sequence of random variables with values in $\{-1,0,1\}$.
There exists a \textit{correspondence} between the classical trinomial model and our jump-binomial model: the role played by the random variables $\Xm_t^{\mathrm{tri}}$ in the classical trinomial model is held in the jump-binomial model by the i.i.d random variables $\eta(t,1)-\eta(t,-1)$. This correspondence can be informally illustrated through the following figures. \\
\begin{center}
\begin{tikzpicture}[scale=0.8]
%
%
 \node[] (RA) at (6,0){\textbf{Trinomial model}};
\node[scale=1] (RA0) at (3,-3.5){$\Sm_{t-1}^{\mathrm{tri}}$};
\node[scale=1] (RA1U) at (8.5,-2){$\Sm_{t}^{\mathrm{tri}}=(1+b)\Sm_{t-1}^{\mathrm{tri}}$};
\draw[->](RA0) -- (RA1U.west) node[sloped,midway,above,scale=0.75]{\textcolor{blue}{$\Xm_t^{\mathrm{tri}}=1$},$\;\textcolor{red}{p_1^{\mathrm{tri}}}$};
\node[scale=1] (RA1) at (8,-3.5){$\Sm_{t}^{\mathrm{tri}}=(1+r)\Sm_{t-1}^{\mathrm{tri}}$}; 
\draw (RA0) -- (RA1) node[sloped,midway,above,scale=0.75]{\textcolor{blue}{$\qquad\Xm_t^{\mathrm{tri}}=0$}};
\draw[->] (RA0) -- (RA1) node[sloped,midway,below,scale=0.75]{\textcolor{red}{$\qquad1-p_1^{\mathrm{tri}}-p_{-1}^{\mathrm{tri}}$}};
\node[scale=1] (RA1D) at (8.5,-5){$\Sm_{t}^{\mathrm{tri}}=(1+a)\Sm_{t-1}^{\mathrm{tri}}$};
\draw[->] (RA0) -- (RA1D.west) node[sloped,midway,below,scale=0.75]{\textcolor{blue}{$\Xm_t^{\mathrm{tri}}=-1$},$\;\textcolor{red}{p_{-1}^{\mathrm{tri}}}$};
\node[] (Xt) at (4.8,-7) {$\Xm_t^{\mathrm{tri}}\in\{-1,0,1\}$};
%
%
\node[] (SA) at (15,0){\textbf{Jump-binomial model} $(\Omega,\cA,(\F_t)_{t\in\id{0}{T}},\gP,\Sm)$};
\node[scale=1] (SA0) at (12,-3.5){$\Sm_{t-1}$};
\node[] (SAs) at (15,-3.5){$\times$};
\node[scale=1] (SA1U) at (15.2,-1){$\Sm_{t}=(1+b)\Sm_{t-1}$};
\draw[->] (15,-3.5)  -- (15,-1.5); 
\node[scale=1] (SA1) at (20,-3.5){$\Sm_{t}=(1+r)\Sm_{t-1}$};
\draw[->] (SA0)  -- (SA1); 
\node[scale=1] (SA1D) at (15.2,-6){$\Sm_{t}=(1+a)\Sm_{t-1}$};
\draw (15,-3.1) --  (15,-1.5) node[sloped,midway,above,scale=0.7]{\textcolor{blue}{$\eta(t,1)=1$}};
\draw (15,-3.1) --  (15,-1.5) node[sloped,midway,below,scale=0.75]{\textcolor{red}{$p$}};
\draw (15,-5.35) --  (15,-3.9) node[sloped,midway,above,scale=0.7]{\textcolor{blue}{$\eta(t,-1)=1$}};
\draw (15,-5.5) --  (15,-3.9) node[sloped,midway,below,scale=0.75]{\textcolor{red}{$1-p$}};
\draw[->] (15,-3.1)  -- (15,-5.5); 
\draw (13.7,-3.5) --  (17.5,-3.5) node[sloped,very near end,color=white,above,scale=0.75]{\textcolor{blue}{$\eta(t,\cdot)=0$}};
\draw (13.7,-3.5) --  (17.5,-3.5) node[sloped,very near end,color=white,below,scale=0.75]{\textcolor{red}{$1-\lambda$}};
\node[] (Xt) at (14.8,-7) {$\eta(t,1)-\eta(t,-1)\in\{-1,0,1\}$};
\end{tikzpicture}
\end{center}
Let us consider a trinomial model defined by the initial value of the asset $\Sm_0^{\mathrm{tri}}$, and $(p_{1}^{\mathrm{tri}},p_{-1}^{\mathrm{tri}})\in(0,1)^2$ such that $p_{1}^{\mathrm{tri}}+p_{-1}^{\mathrm{tri}}<1$.
By setting $\Sm_0=\Sm_0^{\mathrm{tri}}$, $\lambda=p_{-1}^{\mathrm{tri}}+p_{1}^{\mathrm{tri}}>0$ and  $p=p_{1}^{\mathrm{tri}}/\lambda$, we get for all $s\in\R_+^*$,
\begin{equation*}
\mathbf E\Big[s^{\Sm_t/\Sm_{t-1}}\Big]=\esp{s^{1+\theta _t}}
=s^{1+r}(1-\lambda)+s^{1+b}\,\lambda p+s^{1+a}\,\lambda(1-p)=\mathbf E_{\gP^{\mathrm{tri}}}\Big[s^{\Sm_t^{\mathrm{tri}}/\Sm_{t-1}^{\mathrm{tri}}}\Big].
\end{equation*}
All the results in expectation will \textit{de facto} remain valid in the trinomial market model thanks to this identity. 
\end{remark}
\noindent
\underline{\textbf{Martingale measures}} To compute the optimal expected utility in Section 3 via a dual approach, we need to determine the sets of martingale measures in the binomial/jump-binomial models, i.e., the probability measures equivalent to the historical probability measures $\gPb/\gP$ under which $(\St_t)_{t\in\id{0}{T}}$ is a $\F$-martingale. \\
\underline{\textit{In the binomial model}}, $\sP^{\Bi,\F}$ is the set of the $\F$-martingale measures equivalent to $\gPb$. The binomial model stands for  a complete market whose unique risk-neutral probability measure, denoted by $\wPb$, is defined on $(\Omega,\cA)$ by
\begin{equation}\label{risk-neutral_measure_bi_eq}
\wPb(\{\eta(t,\cdot)=1\})=1\quad\text{and}\quad \wPb(\{\eta(t,1)=1\})=(r-a)/(b-a)=:\widehat p.
\end{equation}
Then
\begin{equation*}
\sP^{\Bi,\F}=\{\wPb\}.
\end{equation*}
Under $\wPb$  ("$\mathfrak b$" for \textit{binomial}), the process $\St/\St_{\cdot-1}$ is a binomial/Rademacher process (up to a linear transform).  Define the measures $\wPb_t$ such that $\wPb=:\underset{t\in\id{1}{T}}\bigotimes\;\wPb_t$.
 \vspace{3pt}\\
\underline{\textit{In the jump-binomial model}}, $\sP^{\F}$ is the set $\F$-martingale measures equivalent to $\gP$.
 By virtue of its correspondence with some trinomial market model (see Remark 2.2), we can determine $\sP^{\F}$ by translating the results of Runggaldier \cite{Runggaldier_2006} into our frame. Let us introduce $\gPc$ the measure on $\cA$ such that for all $t\in\id{1}{T}$,
\begin{equation*}
\gPc(\{\eta(t,\cdot)=1\})=0 \quad \text{and}\quad \gP^{\mathfrak c}(\{\eta(t,\pm 1)=1\})=0,
\end{equation*}
the measures $\gPc_t$ such that 
\begin{equation*}
 \gPc=:\underset{t\in\id{1}{T}}\otimes\;\gPc_t.
\end{equation*}
 Note that under $\gPc$ ("$\mathfrak c$" for \textit{constant}),  the process $\St/\St_{\cdot-1}$ is deterministic constant (there is no jump a.s.).\\
In the same vein of  \cite{Runggaldier_2006}, we can prove that $\mathscr P^{\F}$ is the convex hull
\begin{equation*}
\mathscr P^{\F}=\mathrm{Conv}\big\{\gP^j,\;j\in \{1,\dots,2^T\}\big\}
\end{equation*}
whose $2^T$ vertices $\gP^j$ ($j\in \{1,\dots,2^T\}$) are extremal measures such that
\begin{equation}\label{Extremal_measures_eq}
\gP^j=\underset{t\in\id{1}{T}}\bigotimes\;(\wPb_t)^{\gamma_t^j}\,(\gPc_t)^{1-\gamma_t^j},
\end{equation} 
with $\gamma_t^j\in\{0,1\}$ for all $(t,j)\in\id{1}{T}\times\{1,\dots,2^T\}$.
We can note that for all $j\in\{1,\dots,2^T\}$, the $\gP^j$ are not equivalent to $\gP$, since $\gP^j$ coincides on $\sigma(\eta(t,k),k\in\{-1,1\})$ with $\gP_t^{\mathfrak c}$ or $\wPb_t$ which are not. However, any convex combination of the $\gP^j$ is equivalent to $\gP$. 
Let us introduce the probability measure $\wP$ such that
\begin{equation}\label{Def_wP_eq}
\wP=(1-\lambda)\gPc+\lambda\wPb\in\sP^\F,
\end{equation}
where, as a reminder, $\lambda=\gP(\{\eta(1,\cdot)=1\})$.\\
\noindent
As $\lambda$ approaches $0$ (or $1$), the process $\St/\St_{\cdot-1}$ behaves more like a deterministic constant (or binomial) process under $\wP$. This relationship between $\wP$ and the unique risk-neutral measure of the binomial model $\wPb$ will be of crucial importance in solving the utility optimization problems for both the agent and the insider.

\subsection{Clark formula for marked binomial processes}
We recall here the Clark formula for marked binomial processes (see \cite{Halconruy_2021_binomial}), which is given in the case we are interested in here, i.e., when the mark space is reduced to $\{-1,1\}$.  Let us consider $\eta$ a marked binomial process defined on $\Xb$.
\vspace{3pt}\\
\noindent
\underline{\textbf{Functionals}} We denote by $\mathcal L^0(\Omega)$ the class of real-valued measurable functions $\Fm$ on $(\Omega,\cA)$.  For any $\Fm\in\mathcal L^0(\Omega)$, there exists a $\gP$-a.s. unique real-valued measurable function $\fk$ such that $\Fm=\mathfrak f(\eta)$. 
\vspace{5pt}\\
\noindent
\underline{\textbf{The families $\widehat{\mathcal Z}$ and $\widehat{\mathcal R}$}} Let us introduce $\widehat{\mathcal Z}:=\{\D\widehat\Zm_{(t,\pm 1)},\,t\in\id{1}{T}\}$ and $\widehat{\mathcal R}:=\{\D\widehat\Rm_{(t,\pm 1)},\,t\in\id{1}{T}\}$ respectively defined for all $t\in\id{1}{T}$ by
\begin{equation}\label{Def_Z_eq}
\D\widehat\Zm_{(t,1)}=\car_{\{\eta(t,1)=1\}}-\lambda \widehat p\quad \text{and}\quad \D\widehat\Zm_{(t,-1)}=\car_{\{\eta(t,-1)=1\}}-\lambda (1-\widehat p),
\end{equation}
as well as
\begin{equation}\label{Def_R_eq}
 \D\widehat\Rm_{(t,1)}=\D\widehat\Zm_{(t,1)}\quad\text{and}\quad \D\widehat\Rm_{(t,-1)}=\D\widehat\Zm_{(t,-1)}-\rho\Delta\widehat\Rm_{(t,1)}\quad \text{with}\; \rho:=-\frac{\lambda(1-\widehat p)}{1-\lambda \widehat p}.\vspace{5pt}\\
\end{equation}
We can note that the random variables of $\widehat{\mathcal Z}$ and $\widehat{\mathcal R}$ are centred, i.e., for all $t\in\id{1}{T},k\in\{-1,1\}$,
\begin{equation*}
\gE[\D\widehat\Zm_{(t,k)}]=\gE[\D\widehat\Rm_{(t,k)}]=0,
\end{equation*}
and the random variables of $\widehat{\mathcal R}$ are orthogonal with respect to the \textit{mark}, i.e., for all $t\in\id{1}{T},k,\ell \in\{-1,1\}$,
\begin{equation*}
\gE[\D\widehat\Rm_{(t,k)}\D\widehat\Rm_{(t,\ell)}]=\car_{\{k=\ell\}}\gE[\D\widehat\Rm_{(t,k)}^2]
\end{equation*}

\noindent
\underline{\textbf{Malliavin derivative}}
As a reminiscence of the Malliavin operator on the Poisson space, the \textit{add-one cost operator} or \textit{Malliavin's derivative} $\Dm$ is defined  for any $\Fm\in\mathcal L^0(\Omega)$, $t\in\id{1}{T}$ by
\begin{equation}\label{Def_D+_eq}
\Dm_{(t,\pm 1)}\Fm:=\mathfrak f(\pi_{t}(\eta)+\delta_{(t,\pm 1)} )-\mathfrak f(\pi_{t}(\eta)),
\end{equation}
where the map $\pi_t$ is defined for any marked binomial process $\eta$ by
\begin{equation}
\pi_t(\eta)=\displaystyle\sum_{s\neq t}\big[\eta(s,1)+\eta(s,-1)\big].
\end{equation}
For any $t\in\id{1}{T}$, $\Dm_{(t,\pm 1)}\Fm$ measures the effect on $\Fm$ of enforcing the lighting of a point $\pm 1$ at time $t$.\vspace{5pt}\\
\noindent
\underline{\textbf{Clark formula}}
By rewriting Proposition 4.4 of \cite{Halconruy_2021_binomial} into our frame, we get the analogue of the Clark formula: for any $\Fm\in\mathcal L^0(\Omega)$,
\begin{equation}
\Fm=\gE[\Fm]+\sum_{t\in\id{1}{T}}\Big(\gE\big[\Dm_{(t,1)}\Fm\,|\,\F_{t-1}\big]\D\widehat\Rm_{(t,1)}+\gE\big[\Dm_{(t,-1)}\Fm\,|\,\F_{t-1}\big]\D\widehat\Rm_{(t,-1)}\Big).
\end{equation}
\noindent
As a corollary, if $(\Lm_t)_{t\in\id{0}{T}}$ is a $(\gP,\F)$-martingale, for any $(s,t)\in\id{0}{T}^2$ such that $s<t$,
\begin{equation}\label{Clark_formula_mart_eq}
\Lm_t=\Lm_s+\sum_{r=s+1}^t\Big(\gE\big[\Dm_{(r,1)}\Fm\,|\,\F_{r-1}\big]\D\widehat\Rm_{(r,1)}+\gE\big[\Dm_{(r,-1)}\Fm\,|\,\F_{r-1}\big]\D\widehat\Rm_{(r,-1)}\Big).
\end{equation}
\subsection{Enlargement of filtration in a discrete setting: existing results}\label{Subsec_existing_results}
\noindent
The first agent, known as the \textit{ordinary agent}, makes investment decisions based on the publicly available information. On the other hand, the second agent, referred to as the \textit{insider}, possesses additional information right from the start. To distinguish between their information sets, we introduce two separate filtrations: the ordinary agent's information level corresponds to the initial filtration $\F$ (i.e., her knowledge at time $t\in\id{0}{T}$ is given by $\F_{t}$)
whereas the insider disposes at any time $t\in\id{0}{T}$ an information given by the $\sigma$-algebra $\G_{t}$ defined via the initial enlargement 
\begin{equation*}
\G_t=\F_{t}\vee \sigma(\Gm),
\end{equation*}
where $\Gm$ is an $\F_T$-measurable random variable that encodes the information overload enjoyed by the insider.  The random variable $\Gm$ is assumed to fulfill: 
\begin{assumption}\label{Ass_G1} $\Gm$ takes its values in a \textit{finite} 
set $\Gamma$ endowed by a $\sigma$-algebra $\mathscr G$.
\end{assumption}
\noindent
As $\Gm$ takes a finite number of values, even it means removing the values of $\cm$ such that $\mathbf P(\{\Gm=\cm\})=0$, we can consider, that for any for all $\cm\in\Gamma$, $\mathbf P(\{\Gm=\cm\})>0$.
\vspace{3pt}\\
\noindent
In the continuous case, \textit{Jacod's condition} indicates that if the conditional laws of $\Gm$ are absolutely continuous with respect to its law,  then semimartingales are preserved when switching from $\F$ to $\G$. In a discrete setting, Blanchet-Scalliet and Jeanblanc \cite{Blanchet_Jeanblanc_2020} highlight that no such assumption is required and any $(\gP,\F)$-martingale is a $(\gP,\G)$-semimartingale: note that Jacod's hypothesis holds when $\Gm$ takes only discrete values. We recall here important results of Blanchet-Scalliet and Jeanblanc's study (translated into our frame) we will refer as \textbf{\textcolor{teal}{Facts}}  in the sequel. \vspace{2pt}\\\
\noindent
\begin{fact}\textcolor{teal}{(Conditional density process)} (See \cite{Blanchet_Jeanblanc_2020}, Proposition 2.3 (a))\label{Fact_1} 
Under \textbf{Assumption \ref{Ass_G1}}, for any $t<T$ and $\gP$-almost surely for all $\cm\in\Gamma$, we have $\mathbf P(\{\Gm=\cm\}|\F_t)>0$.
\end{fact}
\vspace{3pt}\\
Under \textbf{Assumption \ref{Ass_G1}}, and since $\Gamma=\Gm(\Omega)$ is finite, any set $\Cm\in\mathscr G$ is of the form $\Cm=\bigcup_{\cm\in \Cm}\{\Gm=\cm\}$  and for any $t\in\id{0}{T}$, 
\begin{equation*}
\gP(\{\Gm\in \Cm\}\,|\,\F_{t})=\sum_{\cm\in \Cm}\gP(\{\Gm=\cm\}\,|\,\F_{t})=\sum_{\cm\in \Cm}\frac{\gP(\{\Gm=\cm\}\,|\,\F_{t})}{\gP(\{\Gm=\cm\})}\gP(\{\Gm=\cm\})=:\esp{\q^{\Gm}_t\car_\Cm},
\end{equation*}
where $\q^{\Gm}_t$ is defined by letting for any $\cm\in\Gamma$, $\q_0^{\cm}=1$ and for any $t\in\id{1}{T}$,
\begin{equation}\label{Eq_wq_def}
\q^{\cm}_t=\frac{\gP(\{\Gm=\cm\}\,|\,\F_{t})}{\gP(\{\Gm=\cm\})}.
\end{equation}
\vspace{2pt}\\
Let $(\widehat\Lm_t)_{t\in\id{0}{T}}$ be the density process of $\wP$ (defined by \eqref{Def_wP_eq}) with respect to $\gP$, i.e., such that $\widehat\Lm_0=1$ and $\widehat\Lm_t=(d\wP/d\gP)|\F_t$ for $t\in\id{1}{T}$. Then, for all $t\in\id{0}{T}$, $\widehat\Lm_t$ is not null almost surely and we can define
\begin{equation}\label{Eq_wqG_def}
\widehat\q_t^{\Gm}=\mathbf q_t^{\Gm}/\widehat\Lm_t.
\end{equation}
\\
\noindent
\begin{fact} \label{Fact_2} \textcolor{teal}{(Preservation of semi-martingales)}
(See \cite{Blanchet_Jeanblanc_2020}, Proposition 2.3 (b)) \\
Under \textbf{Assumption \ref{Ass_G1}}, for a given $(\wP,\F)$-martingale $\Xm$, the process $(\Xm_t^{\G})_{t\in\id{1}{T-1}}$ defined by $\Xm_0^{\G}=\Xm_0$, and for any $t\in\id{1}{T-1}$ by
\begin{equation}\label{drift}
\Xm_t^{\G}=\Xm_t-\sum_{s=1}^t\frac{\textless \Xm,\widehat\q^{\cm}\textgreater_s^{\wP}\big|_{\cm=\Gm}}{\widehat\q_{s-1}^{\Gm}}=:\Xm_t-\mu_t^{\G,\Xm},
\end{equation}
is a $(\wP,\G)$-martingale.
\end{fact}
\vspace{2pt}\\
\noindent
\textcolor{white}{space}\\
\noindent
\begin{fact}\textcolor{teal}{($\G$-martingale measures $\Q$ and $\wQ$)}\label{Fact_3a}
(See \cite{Blanchet_Jeanblanc_2020}, Lemma 2.7)  Under \textbf{Assumption \ref{Ass_G1}}, $1/\q^{\Gm}$ is a positive $(\gP,\G)$-martingale on $\id{1}{T-1}$ with expectation $1$.\vspace{3pt}\\
\end{fact}
Then, $1/ \q^{\Gm}$ (and \textit{a fortiori} $1/\widehat \q^{\Gm}$) is  positive, so that we can define $\Q$ and $\wQ$ the probability measures on $(\Omega,\G_{T-1})$ such that for any $\A_t\in\G_t$,
\begin{equation}\label{Eq_gQ_wQ_q_def}
\Q(\A_t)=\mathbf E\big[(1/\q_t^{\Gm})\car_{\A_t}\big] \;\;\text{and}\;\; \wQ(\A_t)=\mathbf E\big[(\widehat\Lm_{T-1}/\q_t^{\Gm})\car_{\A_t}\big]=\mathbf E_{\widehat\gP}\big[(\q_t^{\Gm})^{-1}\car_{\A_t}\big]. 
\end{equation}
\textcolor{white}{space}\\
\noindent
\begin{fact}\textcolor{teal}{(Independence of $\F_t$ and $\G$ under $\wQ$)}\label{Fact_3}
(See \cite{Blanchet_Jeanblanc_2020}, Lemma 2.7) Under \textbf{Assumption \ref{Ass_G1}}, the following statements hold:
\begin{itemize}
\item[(i)] For any $t\in\id{1}{T-1}$, $\F_{t}$ and $\sigma(\Gm)$ are independent under $\wQ$,
\item[(ii)] For any $t\in\id{1}{T-1}$, $\wQ|_{\F_{t}}=\wP|_{\F_{t}}$ and $\wQ|_{\sigma(\Gm)}=\gP|_{\sigma(\Gm)}$.
\end{itemize}
\end{fact}
\textcolor{white}{space}\\
\noindent
\begin{fact}\label{Fact_4} \textcolor{teal}{(Conservation of martingales)}
(See \cite{Blanchet_Jeanblanc_2020}, Proposition 2.6) Under \textbf{Assumption \ref{Ass_G1}}, for any $t\in\id{0}{T-1}$, any $(\wP,\F)$-martingale is a $(\wQ,\G)$-martingale on $\id{0}{t}$.
\end{fact}
\vspace{10pt}\\
We get similar results in $(\Omega,\F,(\G_t)_{t\in\id{1}{T-1}},\gPb)$ as explained in the following part. \vspace{3pt}\\
\underline{\textbf{The processes $\q^{\Bi,\Gm}$ and $\widehat\q^{\Bi,\Gm}$}}\\
For any $t\in\id{1}{T-1}$, 
Jacod's condition holds in $(\Omega,\F,(\G_t)_{t\in\id{1}{T-1}},\gPb)$. To see it, let $(\mathrm M_t)_{t\in\id{0}{T}}$ be the $(\gPb,\F)$-martingale such that $\mathrm M_t=\gPb(\{\Gm=\cm\}|\F_t)$. For any $\cm\in\Gamma$ such that $\gPb(\{\Gm=\cm\})=0$, we have
\begin{equation*}
\mathrm M_T=\gPb(\{\Gm=\cm\}|\F_T)=\mathbf 1_{\{\Gm=\cm\}}=0,\quad\gPb\text{-a.s.},
\end{equation*}
so that we obtain, for $t\in\id{1}{T-1}$, 
\begin{equation*}
\gPb(\{\Gm=\cm\}|\F_t)=\mathrm M_t=\gE[\mathrm M_T|\F_t]=0.
\end{equation*}
The conditional laws of $\Gm$ are then absolutely continuous with respect to its law (under $\gPb$) so that we can define
$\wQb$, the probability measure defined on $(\Omega,\G_{T-1})$  such that for any $t\in\id{0}{T-1}$,
\begin{equation*}
\wQb(\A_t)=\mathbf E\big[(\widehat\Lm^{\Bi}_{T-1}/\q_t^{\Bi,\Gm})\car_{\A_t}\big]=\mathbf E_{\widehat\gP^{\Bi}}\big[(\q_t^{\Bi,\Gm})^{-1}\car_{\A_t}\big]\; ; \; \A_t\in\G_t,
\end{equation*}
where  $\widehat\Lm_t^\Bi=(d\wPb/d\gPb)|\F_t$ and the random variable $\q^{\Bi,\cm}_t $ is defined for any $\omega\in\Omega,\cm\in\Gamma$  by
\begin{equation}\label{proc_PG_Bi}
 \q_t^{\Bi,\cm}(\omega)=\frac{\gPb(\{\Gm=\cm\}\,|\F_{t})(\omega)}{\gPb(\{\Gm=\cm\})}.
\end{equation}
Let also $(1/\widehat\q^{\Bi,\Gm})$ be the $\G$-adapted process such that for $t\in\id{1}{T-1}$, $1/\widehat\q_t^{\Bi,\Gm}:=\widehat\Lm_t^{\Bi}/\q_t^{\Gm}$. We can state the analogues of \textcolor{teal}{\textbf{Fact 4}} and \textcolor{teal}{\textbf{Fact 5}} in $(\Omega,\cA,(\G_t)_{t\in\id{1}{T-1}},\gPb)$ by replacing everywhere needed $\gP$, $\wP$, $\q^{\Gm}$, $\widehat\q^{\Gm}$ respectively by $\gPb$, $\wPb$, $\q^{\Bi,\Gm}$, $\widehat\q^{\Bi,\Gm}$.
\section{Insider vs agent: the rewards of extra information}\label{Section_utility}
For the sake of simplicity, \textbf{we assume in this section that $r=0$} and we work directly with discounted prices. \\
In this section, we compute and compare the maximum expected utility of both the ordinary agent and the insider, in order to quantify the latter's edge and measure the benefit of the additional information at her hands. 
\subsection{Utility maximization problems: setting and notation}\label{Subsec_utility_notation}

\subsubsection*{Portfolios and strategies}
We consider an economic agent and an insider both disposing of $x\in\R_+^*$ euros at date $t=0$ (initial budget constraint), for whom we want to determine the maximal expected logarithmic, exponential and power utilities (defined below) from terminal wealth. Let $\HH$ be some filtration on $(\Omega,\cA)$, that may and shall be replaced by $\F$ or $\G$ later on. As a reminder, the value of a $\HH$-portfolio at time $t\in\id{0}{T}$ is given by the random variable
\begin{equation*}
\Vm_t(\psi)=\alpha_{t}\,+\vp_{t}\,\Sm_t,
\end{equation*}
where the so-called $\HH$-\textit{strategy} $\psi=(\alpha_t,\vp_t)_{t\in \id{0}{T}}$ with initial value $(\alpha_0,\varphi_0)$ is a couple of $\HH$-predictable processes modelling respectively the amounts of riskless and risky assets held in the portfolio. Without loss of generality, we may and shall assume that $\varphi_0=0$. A $\HH$-strategy $\psi=(\alpha,\varphi)$  is said to be \textit{self-financing} if it fulfils the condition:  
\begin{equation}\label{Self_financing_cond}
(\alpha_{t+1}-\alpha_{t})+\Sm_t\,(\vp_{t+1}-\vp_{t})=0,
\end{equation}
for any $t\in \id{1}{T-1}$.  
A nonnegative $\HH_{T}$-measurable random variable  $\Fm$ (called \textit{claim}) is \textit{replicable} or \textit{reachable} if there exists an $\HH$-predictable self-financing strategy $\psi=(\alpha,\varphi)$ which corresponding portfolio value satisfies $
\alpha_0=\Vm_0(\psi)>0$ and  $\Vm_{T}(\psi)=\Fm$.  Let $\mathscr S_{\HH}(x)$ be the class of \textit{$\HH$-admissible strategies} of initial value $x$, i.e.,
\begin{equation}\label{admissible_strat_eq}
\sS_{\HH}(x)=\{\psi=(\alpha,\varphi)\,|\,\alpha_0=x,\,\varphi\;\text{is}\,\HH\text{-predictable,\,\,}\psi \text{\,is\,self-financing\,and\,}\Vm_t(\psi)> 0,\,\forall t\in\id{0}{T} \}.
\end{equation}
\subsubsection*{Utility maximization problems}
Let $x\in\R_+^*$. In this section, we are led to consider the optimization problems from the agent's point of view at any time $t\in\id{1}{T}$,
\begin{equation}\label{optTerAg_eq}
\Phi_t^{\F,u}(x)=\underset{\psi\,\in\sS_{\F}(x)}\sup\esp{u(\Vm_{t}(\psi))},
\end{equation} 
and from the  insider's at any time $t\in\id{1}{T-1}$,
\begin{equation}\label{optTerIns_eq}
\Phi_t^{\G,u}(x)=\underset{\psi\,\in\sS_{\G}(x)}\sup\esp{u(\Vm_{t}(\psi))},
\end{equation} 
where $u$ is a \textit{utility function}, strictly increasing and strictly concave on $\R$ or $\R_+^*$. Throughout, we could consider utility functions $u$ that can be \textit{logarithmic}, \textit{exponential} or a \textit{power function}. For each one, we designate its conjugate function by $v^{u}$:
\begin{itemize}
\item \textbf{Logarithmic utility} (as $\mathbf{log}$) $u:x\in\R_+^*\mapsto\log(x)$, $v^{\mathbf{log}}:y\in\R_+^*\mapsto-\log(y)-1$.
\item \textbf{Exponential utility} (as $\mathbf{exp}$) $u:x\in\R\mapsto -\exp(-x)$, $v^{\mathbf{exp}}:y\in\R_+^*\mapsto y(\log(y)-1)$.
\item \textbf{Power utility} (as $\mathbf{pow}$) $u:x\in\R_+^*\mapsto x^{\alpha}/\alpha$ (with $\alpha\in(0,1)$), $v^{\mathbf{pow}}:y\in\R_+^*\mapsto -(1/\beta)y^{\beta}$ with $\beta=\alpha/(\alpha-1)$.
\end{itemize}
\subsubsection*{Dual optimization problems} In the sequel, we solve \eqref{optTerAg_eq} and \eqref{optTerIns_eq} by a dual approach that can be found in Delbaen and Schachermayer (\cite{Delbaen_Schachermayer}, section 3). For $t\in\id{1}{T}$, this boils down for the agent
\begin{equation}\label{dual_Ag_eq}
\Psi_t^{\F,u}(y)=\underset{\M\,\in\sP^{\F}}\inf\gE_{\M}\Big[v^{u}\Big(y\frac{d\mathbf M}{d\gP}\Big|_{\F_t}\Big)\Big],
\end{equation}
where $\sP^{\F}$ is the set of $\F$-martingale measures equivalent to $\gP$. Note that solving \eqref{optTerIns_eq} (for the insider) via a dual approach means optimizing with respect to $\sP^{\G}$, namely the set of $\G$-martingale measures equivalent to $\gP$ on $\id{1}{T-1}$. In order to explain our approach, let us recall the (martingale) probability measures (PM) we handle in the binomial and jump-binomial models, for the agent and the insider.
\begin{center}
\captionof{table}{(Martingale) Probability Measures (PM)}\label{Tab_chaotic_dec}
\begin{tabular}{|c|l|l|l|}
\hline
\textbf{Model}& \hspace{30pt}Binomial model  & \hspace{50pt}Jump-binomial model\\
$\backslash$ &\hspace{20pt}$(\Omega,\cA,(\F_t)_{t\in\id{0}{T}},\gPb,\Sm)$&\hspace{50pt}$(\Omega,\cA,(\F_t)_{t\in\id{0}{T}},\gP,\Sm)$\\
\textbf{Investor} & & \\
\hline
\hline
&\underline{Historical PM} $\gPb$& \underline{Historical PM} $\gP$\\
& $\gPb(\eta(1,\cdot)=1)=1$& $\gP(\eta(1,\cdot)=1)=\lambda$  \\
& $\gPb(\eta(1,1)=1)=p$& $\gP(\eta(1,1)=1)=\lambda p$  \\
\textbf{Agent} & &\\
&\underline{Martingale PM set} $\sP^{\Bi,\F}=\{\wPb\}$&  \underline{Martingale PM set} $\sP^{\F}=\mathrm{Conv}\{\gP^j,\,j\in\id{1}{2^T}\}$\\
& $\wPb(\eta(1,\cdot)=1)=1$&  \\
& $\wPb(\eta(1,1)=1)=\widehat p$& e.g. $\wP=(1-\lambda)\gP^{\mathfrak c}+\lambda \wPb\in\sP^{\F}$  \\
&&\\
\hline
&\underline{Special martingale PM} $\wQb$&  \underline{Special martingale PM} $\wQ$\\
&$(d\wQb/d\gPb)|\G_t=\widehat\Lm_t^\Bi/\q_t^{\mathfrak b,\Gm}$& $(d\wQ/d\gP)|\G_t=\widehat\Lm_t/\q_t^{\Gm}$\\
\textbf{Insider} &&\\
&\underline{Martingale PM set} $\sP^{\Bi,\G}$&  \underline{Martingale PM set} $\sP^{\G}$\\
&determined in subsection \ref{Subsec_set_PBG} & difficult to determine\\
&&\\
\hline
\end{tabular}
\end{center}
Solving the dual problems \eqref{optTerAg_eq} and \eqref{optTerIns_eq} directly for the agent and the insider involves dealing with sets of probability measures that can be quite large (e.g., $\sP^{\F}$ for the agent) or very challenging to describe (e.g., $\sP^{\G}$ for the insider). However, we overcome this difficulty by leveraging the \textit{volatility} structure of the jump binomial model. This allows us to simplify the dual problems to the binomial model, where the martingale measure set $\sP^{\Bi}$ (for the agent) reduces to $\sP^{\Bi,\F}$, while $\sP^{\Bi,\G}$ (for the insider) can be described as shown in subsection \ref{Subsec_set_PBG}.
\color{black}
\subsection{Agent's maximum expected utility}\label{Subsec_Insider_max}
The analogue of the maximization problem \eqref{optTerAg_eq} can be elegantly solved in the trinomial model viewed as an embryonic \textit{volatility} model. This idea is in line with observations in Runggaldier \textit{et al.} \cite{Runggaldier_Trivellato_Vargiolu_2002}, in Vargiolu (\cite{Vargiolu_2002}, remark 3) or in Delbaen and Schachermayer (\cite{Delbaen_Schachermayer}, section 3.3). An \textit{underlying volatility structure} clearly appears in the construction of our jump-binomial model itself. To illustrate this, let's consider the simple one-period case.
\subsubsection*{Toy example $T=1$}
A basic computation leads to 
\begin{align*}
\gE\big[u(\Vm_{T}(\psi))\big]
&=(1-\lambda)\gE\big[u(x+\varphi_T\D\Sm_T)\,|\,\eta(T,\cdot)=0\big]+\lambda\gE\big[u(x+\varphi_T\D\Sm_T)\,|\,\eta(T,\cdot)=1\big]\\
&=(1-\lambda)u(x)+\lambda\mathbf E_{\gPb}\big[u(\Vm_{T}(\psi))\big].
\end{align*}
Then, the optimal strategy for $\Phi_T^{\F,u}$ is the same as the unique optimizing strategy solution of the optimization problem $\Phi_T^{\mathfrak b,\F,u}(x)$ defined by
\begin{equation}\label{opt_Bi_eq}
\Phi_T^{\mathfrak b,\F,u}(x):=\underset{\psi\,\in\sS_{\F}(x)}\sup\mathbf E_{\gPb}\big[u(\Vm_{T}(\psi))\big].
\end{equation}
Note that $\Phi_T^{\F,\Bi,u}(x)$  (for $T=1$) can be solved by considering its dual problem
\begin{equation}\label{dual_optTerAg_eq}
\Psi_T^{\Bi,u}(y)=\underset{\M\,\in\sP^{\Bi,\F}}\inf\gE_{\M}\Big[v^{u}\Big(y\frac{d\mathbf M}{d\gP}\Big)\Big]=\gE_{\wPb}\Big[v^{u}\Big(y\frac{d\wPb}{d\gP}\Big)\Big]=:\gE_{\wPb}\Big[\widehat{\Vm}_{T}^{\Bi,\F,u}\Big],
\end{equation}
since the set $\sP^{\Bi,\F}$ of the $\F$-martingale measures equivalent to $\gPb$ is the reduced to $\{\wPb\}$.\vspace{3pt}\\
We can then deduce the procedure:
\subsubsection*{Agent's utility optimization procedure}
\begin{enumerate}
\item Solve the $u$-utility optimization problem $\Phi_T^{\F,\Bi,u}(x)$ for $T=1$ by considering its dual problem $\Psi_T^{\Bi,u}(y)$.
This provides the optimal (discounted) portfolio value $\widehat{\Vm}_{T}^{\Bi,\F,u}$ in terms of $\gPb$ and $\wPb$.
\item Deduce the optimal discounted portfolio value $\widehat{\Vm}_{T}^{\F,u}$ for the jump-binomial model with one period by replacing $\gPb$ and $\wPb$ respectively by $\gP$ and $\wP$, where, as a reminder, $\wP$ is the element of $\sP^\F$ defined by \eqref{Def_wP_eq}, i.e., $\wP=(1-\lambda)\gPc+\lambda\wPb$.
\item Extend the results at any time $t\in\id{1}{T}$. Since the increments of the stock price process are i.i.d., this can be achieved through the usual dynamic programming method (i.e., a backward induction process).
\end{enumerate}

By following the procedure described above, we translate some results from Delbaen and Schachermayer (\cite{Delbaen_Schachermayer}, section 3) or Pascucci and Runggaldier (\cite{Pascucci_Runggaldier_2012}, section 2.4) into the jump-binomial model. We retrieve the formulas stated in the binomial model for logarithmic, exponential, and power utility functions, as well as in the trinomial model for power utility in the one-period case (see \cite{Delbaen_Schachermayer}). Additionally, we obtain new formulas for exponential and power utilities in the multi-period case. \vspace{3pt}\\
\noindent
We need the following definition: Given two probability measures defined on the same measurable space $(\Omega,\HH)$ where $\cB$ is a $\sigma$-algebra, $\mathfrak D_{\cB}(\gP||\Q)$ designates the \textit{Kullback-Leibler divergence} or \textit{relative entropy} of $\gP$ with respect to $\Q$ on $\cB$ and is defined by
\begin{displaymath} 
\mathfrak D_{\HH}(\gP||\Q)=\left\{
\begin{array}{lll}
&\esp{\log\bigg(\dfrac{d\gP}{d\Q}\bigg|_{\cB}\bigg)} & \quad \text{if} \quad \gP\ll\Q \; \text{on} \; \cB\vspace{2pt}\\
& +\infty &\quad \text{otherwise}.
\end{array}\right.\\
\end{displaymath}
\noindent
Note that by definition of $\wP$ \eqref{Def_wP_eq} we have 
\begin{equation}\label{Kullback-Leibler_rem_eq}
\mathfrak D_{\F}(\gPb||\wPb)=\mathfrak D_{\F}(\gP||\wP).
\end{equation}

\begin{proposition}[Agent's portfolio optimization]\label{Opt_port_ag_ter_T_prop}
For $x\in\R_+^*$, $t\in\id{1}{T}$ and $u\in\{\mathbf{log},\mathbf{exp},\mathbf{pow}\}$,  let $\widehat{\Vm}_{t}^{\F,u}$ be the optimal portfolio (discounted) value for the problem $\Phi_t^{\F,u}(x)$ defined by \eqref{optTerAg_eq}. We get:
\vspace{5pt}\\
\noindent
\underline{\textbf{Logarithmic utility}}: $$\widehat{\Vm}_{t}^{\F,\mathbf{log}}=x\cdot\dfrac{d\gP}{d\gPs}\bigg|_{\F_t}.$$ 
%
\noindent
\underline{\textbf{Exponential utility}}:
$$\widehat{\Vm}_{t}^{\F,\mathbf{exp}}=x+\mathfrak D_{\F_t}(\gPs||\gP)+\log\Big(\dfrac{d\gP}{d\gPs}\Big|_{\F_t}\Big).$$
\vspace{3pt}\\
\noindent 
\underline{\textbf{Power utility}}: Let $\widehat \Lm_t=(d\wP/d\gP)|\F_t$ and $\beta=\alpha/(\alpha-1).$ $$\widehat{\Vm}_{t}^{\F,\mathbf{pow}}=x\cdot\gE\big[\widehat\Lm_t^{\beta}\big]^{-1}\cdot\bigg(\dfrac{d\gPs}{d\gP}\Big|_{\F_t}\bigg)^{\beta-1}.$$
\end{proposition}
\noindent
We can check that $\widehat{\Vm}_{x,0}^{\F,u}=x$ for all $u\in\{\mathbf{log},\mathbf{exp},\mathbf{pow}\}$.

\subsection{Insider's maximum expected utility}\label{Agent_ins_T_subsec}
In this subsection, we address the utility optimization problem \eqref{optTerIns_eq} for the insider. As mentioned earlier, since similar arguments apply in the context of the insider, the optimal strategy for the insider in the jump-binomial model is the same as the one obtained in the binomial model.
Thus we can derive $\Phi_t^{\G,u}(x)$ ($t\in\id{1}{T-1}$) from the problem
\begin{equation}\label{opt_Bi_ins_eq}
\Phi_t^{\Bi,\G,u}(x):=\underset{\psi\,\in\sS_{\G}(x)}\sup\mathbf E_{\gPb}\big[u(\Vm_{t}(\psi))\big].
\end{equation}
However, contrary to agent's paradigm, the set of $\G$-martingale measures equivalent to $\gPb$ on $\id{1}{T-1}$ is not reduced to a single element. To solve \eqref{opt_Bi_ins_eq} we are led to consider its dual problem $\Psi_t^{\Bi,\G,u}(y)$ defined by
\begin{equation}\label{dual_Ins_eq}
\Psi_t^{\G,u}(y)=\underset{\M\,\in\sP^{\G}}\inf\gE_{\M}\Big[v^{u}\Big(y\frac{d\mathbf M}{d\gP}\Big|_{\G_t}\Big)\Big],
\end{equation}
where $\mathscr P^{\mathfrak b,\G}$ is the set of $\G$-martingale measures equivalent to $\gPb$ on $\id{1}{T-1}$. We need to determine it, which is the purpose of the following subsection.
\subsubsection*{Martingale measures for the insider: the set $\sP^{\Bi,\G}$}\label{Subsec_set_PBG}
To describe the set $\sP^{\Bi,\G}$, we use an argument of Grorud and Pontier \cite{Grorud_Pontier_1999} provided the market is \textit{complete} for the insider in the following sense: any $\G_{T-1}$-measurable bounded contingent claim $\Fm$ can be hedged by a strategy in $\sS^{\G}$.
To prove it, we show that $\Sm$ satisfies a $\G$-predictable representation property in $(\Omega,\cA,(\G_t)_{t\in\id{1}{T-1}},\wPb)$. \vspace{3pt}\\
We begin with this handy technical lemma that we will also use in subsection \ref{Subsec_insider_hedging}.
\begin{lemma}\label{Lem_DetltaS}
For any $t\in\id{1}{T-1}$, $k\in\{-1,1\}$,
\begin{equation}\label{Eq_wQ_wP_equal_as}
\wQ(\{\eta(t,k)=1\}|\G_{t-1})=\lambda\widehat p, \quad\gP\mathrm{-a.s.}
\end{equation}
As a consequence, for any $t\in\id{1}{T-1}$, 
\begin{equation}\label{Eq_DetltaS}
\frac{\D\St_t}{\St_{t-1}}=\frac{b-r}{1+r}\Delta\widehat\Zm_{(t,1)}+\frac{a-r}{1+r}\Delta\widehat\Zm_{(t,-1)}.
\end{equation}
\end{lemma}
\noindent
Note that we can state an analogue property for $\wQb$: 
For any $t\in\id{1}{T-1}$, $k\in\{-1,1\}$,
\begin{equation}\label{Eq_wQb_wPb_equal_as}
\wQb(\{\eta(t,k)=1\}|\G_{t-1})=\widehat p, \quad\gPb\mathrm{-a.s.}
\end{equation}
\underline{\textbf{Predictable representation property}} As a reminder, $\widehat \gP^\Bi$ is the risk-neutral probability measure in $(\Omega,\cA,\F,\gP^{\Bi},\Sm)$ and is defined by \eqref{risk-neutral_measure_bi_eq}.
For any $t\in\id{1}{T-1}$, let us first define $\D\overline\Zm_t^{\Bi}$ by
\begin{equation}\label{Def_Z_Bi_norm_eq}
\D\overline\Zm_t^{\Bi}=[\lambda\widehat p(1-\lambda\widehat p)]^{-1/2}\big(\mathbf 1_{\{\eta(t,\cdot)=1\}}-\lambda).
\end{equation}
As a consequence of \eqref{Eq_wQb_wPb_equal_as}, we can check that for all $t\in\id{1}{T-1}$,
\begin{equation*}
\gE_{\wQb}[\D\overline\Zm_t^{\Bi}|\G_{t-1}]=0\quad\text{and}\quad\gE_{\wQb}[(\D\overline\Zm_t^{\Bi})^2|\G_{t-1}]=1.
\end{equation*}
Then the family $\{\D\overline\Zm_t^{\Bi},\,t\in\id{1}{T-1}\}$ stands for the analogue of the (Rademacher) \textit{structure equation solution} (see Privault \cite{Privault_stochastica}, section 1.4) in $(\Omega,\cA,(\G_t)_{t\in\id{1}{T-1}},\wQb)$.
Moreover, it drives the dynamics of $\Sm$: For all $t\in\id{1}{T-1}$, 
\begin{align*}
\D\Sm_t
&=\Sm_{t-1}\Big[(1+b)\car_{\{\eta(t,1)=1\}}+(1+a)\car_{\{\eta(t,-1)=1\}}-1\Big]\\
&=\Sm_{t-1}\Big[(b-a)\big(\car_{\{\eta(t,1)=1\}}-\widehat p\big)+\widehat p(b-a)+a\Big]=\Sm_{t-1}(b-a)\D\widehat\Zm_t^{\Bi},
\end{align*}
since, by \textcolor{teal}{\textbf{Fact 4}} (subsection 2.2), $\Sm$ is a $(\wQb,\G)$-martingale on $\id{1}{T-1}$ and then $\widehat p(b-a)+a-r=0$. Then, it follows from Privault (\cite{Privault_stochastica}, Proposition 1.7.5) that for any $\G_t$-measurable random variable $\Fm$ there exists a $\G$-predictable process $\psi$ such that 
\begin{equation}\label{predictable_binomial_eq}
\Fm=\gE_{\wQb}[\Fm|\G_0]+\sum_{s=1}^t\psi_s\D\overline\Zm_s^{\Bi}=\gE_{\wQb}[\Fm|\G_0]+\sum_{s=1}^t\frac{\psi_s[\lambda(1-\lambda)]^{1/2}}{\Sm_{s-1}(b-a)}\D\Sm_s.
\end{equation}
Checking that the process $\varphi:=[\psi[\lambda(1-\lambda)]^{1/2}]/[{(b-a)\Sm_{\cdot-1}}]$ is $\G$-predictable, we deduce that $\Sm$ has the predictable representation property in $(\Omega,\cA,(\G_t)_{t\in\id{1}{T-1}},\wQb)$, i.e., that any $\G_t$-measurable random variable $\Fm$ can be represented as
\begin{equation*}
\Fm=\gE_{\wQb}[\Fm|\G_0]+\sum_{s=1}^t\varphi_s\D\Sm_s,
\end{equation*}  
where $\varphi=(\varphi_s)_{s\in\id{0}{t}}$ is a $\G$-predictable process. Then, the binomial model market is \textit{complete} for the insider.\vspace{3pt}\\
\underline{\textbf{The set $\sP^{\Bi,\G}$}}
can be then obtained using a result from Grorud and Pontier \cite{Grorud_Pontier_1999}: as the market is \textit{complete} for the insider, the set $\sP^{\Bi,\G}$  writes
\begin{equation}\label{Set_P_Bi-G_eq}
\sP^{\Bi,\G}=\big\{\Um*\wQb,\, \Um\,\in\mathscr{U}^{\Bi,\G} \big\},
\end{equation}
where $\mathscr{U}^{\Bi,\G}$ is the set of $\sigma(\Gm)$-measurable ($\G_0=\sigma(\Gm)$) positive random variables $\Um$  such that $\gE_{\wQb}[\Um]=1$. We use the notation $*$ to indicate that $\Um$ is the Radon-Nikodym derivative of the probability measure $\bM:=\Um*\wQb$ with respect to $\wQb$.

\subsubsection*{Insider's utility optimization in the (jump-)binomial model}
As for the agent (subsection 2.2), we will first solve the associated dual problem in the binomial model \eqref{dual_Ins_eq}. We adapt Theorem 3.2.1 in \cite{Delbaen_Schachermayer} into our frame.
In the same vein, we define for $u\in\{\mathbf{log},\mathbf{exp},\mathbf{pow}\}$,
\begin{equation}\label{U_opt_ins_def_eq}
\Um^{u}:=\underset{\Um\in\mathscr U^{\mathfrak b,\G}}{\mathrm{argmin}}\,\left\{\gE\Big[v\Big(y^{u}\frac{d[\Um*\widehat{\mathbf Q}^{\mathfrak b}]}{d\mathbf P^{\mathfrak b}}\Big|_{\G_{T-1}}\Big)\Big]+xy^{u}\right\}
\end{equation}
where $y^{\mathbf{log}}=1/x$, $y^{\mathbf{exp}}=\exp(-x-\mathfrak D_{\G_{T-1}}(\wQb||\gPb))$, $y^{\mathbf{pow}}=x^{1/(\beta-1)}\gE\big[(d\wQb/d\gPb)^{\beta}\big]^{-1}$.
\vspace{3pt}\\

Let us introduce $\sP^{\mathrm{opt},\G}$, which we refer to as \textit{the insider's optimal measure set:}
\begin{equation*}\label{Eq_insider_opt_measure_set}
\sP^{\mathrm{opt},\G}=\big\{(1-\lambda)\gPc+\lambda(\Um*\wQb),\, \Um\,\in\mathscr{U}^{\Bi,\G} \big\}.
\end{equation*}
We can now state our first main result: the explicit solution of the insider utility maximization problem in the jump-binomial model. 
\begin{theorem}[Insider's utility optimization in the jump-binomial model]\label{Opt_port_ins_ter_T_prop}
For $x\in\R_+^*$, $t\in\id{1}{T-1}$ and $u\in\{\mathbf{log},\mathbf{exp},\mathbf{pow}\}$, let $\widehat{\Vm}_{t}^{\G,u}$ be the optimal portfolio (discounted) value for the problem $\Phi_t^{\G,u}(x)$ defined by \eqref{optTerIns_eq}. Let us
define $\widehat\Q^{u}$ the probability measure equivalent to $\gP$ such that 
\begin{equation*}
\widehat\Q^{u}=(1-\lambda)\gPc+\lambda(\Um^{u}*\wQb)\in\sP^{\mathrm{opt},\G},
\end{equation*}
where $\Um^{u}$ is defined by \eqref{U_opt_ins_def_eq}. We get:\vspace{3pt}\\
\noindent
\underline{\textbf{Logarithmic utility}}: 
\begin{equation}\label{Eq_opt_insider_log}
\widehat{\Vm}_{t}^{\G,\mathbf{log}}=x\cdot\dfrac{d\gP}{d\wQ^{\mathbf{log}}}\bigg|_{\G_t}.
\end{equation}
%
%
\noindent
\underline{\textbf{Exponential utility}}: 
\begin{equation*}
\widehat{\Vm}_{t}^{\G,\mathbf{exp}}=x+\mathfrak D_{\G_t}(\widehat\Q^{\mathbf{exp}}||\gP)+\log\Big(\dfrac{d\gP}{d\widehat\Q^{\mathbf{exp}}}\Big|_{\G_t}\Big).
\end{equation*}
\noindent 
\underline{\textbf{Power utility}}: Let $\beta=\alpha/(\alpha-1)$.
\begin{equation*}
\widehat{\Vm}_{t}^{\G,\mathbf{pow}}=x\cdot\gE\Big[\Big(\dfrac{d\widehat\Q^{\mathbf{pow}}}{d\gP}\Big|_{\G_t}\Big)^{\beta}\Big]^{-1}\cdot\bigg(\dfrac{d\widehat\Q^{\mathbf{pow}}}{d\gP}\Big|_{\G_t}\bigg)^{\beta-1}.
\end{equation*}
\end{theorem}
\subsection{Insider's advantage and impact of the extra information}
\noindent
The insider's additional expected $u$-utility for $u\in\{\mathbf{log},\mathbf{exp},\mathbf{pow}\}$ and up to time $t\in \id{1}{T-1}$ is defined by
\begin{equation*}
\mathcal U_t^{u}(x)=\underset{\psi\,\in\mathscr S_{\G}(x)}\sup\esp{u(\Vm_{t}(\psi))}-
\underset{\psi\,\in\mathscr S_{\F}(x)}\sup\esp{u(\Vm_{t}(\psi))}.
\end{equation*} 
\noindent
Let us define for any $t\in\id{1}{T-1}$, 
$\mathrm{Ent}(\Gm)$ and $\mathrm{Ent}(\Gm\,|\,\HH_t)$ by 
\begin{equation*}
\mathrm{Ent}(\Gm)=-\sum_{\cm\in\Gamma}\log\big(\gP(\{\Gm=\cm\})\big)\,\gP(\{\Gm=\cm\}),
\end{equation*}
and
\begin{equation*}
\mathrm{Ent}(\Gm\,|\,\HH_t)=-\esp{\sum_{\cm\in\Gamma}\log\big(\gP(\{\Gm=\cm\}\,|\,\HH_t)\big)\,\gP(\{\Gm=\cm\}\,|\,\HH_t)},
\end{equation*}
that respectively stand for the \textit{entropy} of the random variable $\Gm$ and its \textit{conditional entropy} with respect to the filtration $\HH$.\\ 
To our knowledge, the computation of the insider's additional expected utility has been limited to the logarithmic case and continuous-time complete market models in \cite{Amendinger_Imkeller_Schweizer_1998}. However, our results for exponential and power utilities are new. This presents our second main result, which is also the most significant one in this section.
\begin{theorem}\label{additional_utility_th}
Assume that the ordinary agent and the insider have an initial budget $x\in\R_+^*$.
For $u\in\{\mathbf{log},\mathbf{exp},\mathbf{pow}\}$, the insider's additional expected $u$-utility up to time $t\in \id{1}{T-1}$ is given by:
\\
\noindent
\underline{\textbf{Logarithmic utility}}: 
\begin{equation}\label{additional_utility_eq}
\mathcal U_t^{\mathbf{log}}(x)=\lambda\mathfrak D_{\G_t}(\widehat\gP||\widehat\Q)=\mathrm{Ent}(\Gm)-\mathrm{Ent}(\Gm\,|\,\F_t)-\gE_{\gPb}[\log(\Um^{\mathbf{log}})].
\end{equation}
\noindent 
\underline{\textbf{Exponential utility}}: 
\begin{equation*}
\mathcal U_t^{\mathbf{exp}}(x)=-\exp\big(-x\mathfrak D_{\G_t}(\widehat\Q||\gP)\big)+\exp\big(-x\mathfrak D_{\F_t}(\widehat\gP||\gP)\big).
\end{equation*}
\noindent 
\underline{\textbf{Power utility}}: 
\begin{equation*}
\mathcal U_t^{\mathbf{pow}}(x)
=\frac{\lambda x^{\alpha}}{\alpha}\bigg(\gE\bigg[\Big(\frac{d\widehat\Q^{\pow}}{d\gP}\Big|_{\G_t}\Big)^{\beta}\bigg]^{1-\alpha}-\gE\Big[\big(\widehat\Lm_t\big)^{\beta}\Big]^{1-\alpha}\bigg).
\end{equation*}
\end{theorem}

\noindent
Since  for all $u\in\{\mathbf{log},\mathbf{exp},\mathbf{pow}\}$ $\Um^{u}$ is a $\sigma(\Gm)$-measurable random variable such that $\mathbf E_{\wQb}[\Um^{u}]=1$, we can check that $\widehat{\Vm}_{0}^{\G,u}=x$.
\noindent
\begin{remark}
We obtain the discrete counterpart of Theorem 4.1 in Amendinger, Imkeller, and Schweizer \cite{Amendinger_Imkeller_Schweizer_1998}, which holds for the Black-Scholes model and a discrete random variable $\Gm$. Our result expresses the additional expected logarithmic utility of the insider in terms of the relative entropy of $\Gm$.
Furthermore, our findings can be compared to Theorem 5.12 in Ankirchner \textit{et al.} \cite{Ankirchner_Dereich_Imkeller_2006}, where they establish that, under an initial enlargement (continuous) setting, the insider's additional utility is related to the \textit{relative difference} of the enlarged filtration with respect to the initial one. In fact, this also coincides with the Shannon entropy between (with the corresponding notations) $\Gm$ and some random variable $\mathrm{Id}_{\F_T}$.
However, in the continuous case, the result still holds at the deadline $T$ by taking the limit as $t$ approaches $T$. In our discrete framework, this is not the case due to the existence of an arbitrage opportunity at the horizon $T$, causing the insider's utility gain to become infinite at that time.
\end{remark}

Similar to the logarithmic case, we can express the additional expected exponential and power utilities in terms of the (conditional) entropy of $\Gm$. Essentially, the insider's advantage can be quantified by the entropy of the random variable G, reflecting the information she has from prior knowledge of G's outcome. The greater the difference between entropy and conditional entropy of $\Gm$, the larger the insider's additional utility. The following estimates are obtained:
\begin{corollary}\label{additional_utility_cor} Under the same assumptions as Theorem \ref{additional_utility_th}, for $t\in\id{1}{T-1}$, we have the following bounds:\\
\noindent 
\underline{\textbf{Exponential utility}}: There exists $\kappa\in(0,1)$ and a probability measure $\mathbf{M}_\kappa=(1-\kappa)\widehat\Q+\kappa\widehat\gP$ such that
\begin{equation*}
\mathcal U_t^{\mathbf{exp}}(x)\leqslant \exp\big(-x\mathfrak D_{\G_t}(\mathbf{M}_\kappa||\gP)\big)\big[\mathrm{Ent}(\Gm)-\mathrm{Ent}(\Gm\,|\,\F_t)-\gE_{\gPb}[\log(\Um^{\mathbf{log}})]\big].\\
\end{equation*}
\noindent 
\underline{\textbf{Power utility}}: There exists $\kappa\in(0,1)$ and a random variable satisfying $\log(\widehat\Lm_t^{\mathbf{pow},\kappa}):=(1-\kappa)\log(\widehat\Lm_t)+\kappa\log(\widehat\Lm_t^{\mathbf{pow}})$ such that
\begin{equation*}
\mathcal U_t^{\mathbf{pow}}(x)\leqslant\frac{|\beta|^{1-\alpha}x^{\alpha}}{\alpha}\Big\|\big(\widehat\Lm_t^{\mathbf{pow},\kappa}\big)^{\beta}\Big\|_{\infty}^{1-\alpha}\big[\mathrm{Ent}(\Gm)-\mathrm{Ent}(\Gm\,|\,\F_t)-\gE_{\gPb}[\log(\Um^{\mathbf{pow}})]\big]^{1-\alpha}.
\end{equation*}
\end{corollary}

\section{Inside insider's mind}

In this section, we use the Malliavin calculus for marked binomial processes to state to establish two results of interest from the insider's perspective: a new interpretation of the information drift and the computation of the insider's \textit{optimal} hedging strategy.

\subsection{A new interpretation of the \textit{information drift}
}
By \textcolor{teal}{\textbf{Fact 2}}, martingales with respect to the initial filtration become semimartingales by moving to the enlarged one. This transfer is encoded by a particular process $\mu^{\G}$, called the \textit{information drift}, i.e., the drift to eliminate so that the price dynamics remains a martingale from insider's point of view (see definition in \cite{Ankirchner_Dereich_Imkeller_2006}). Note that by \textcolor{teal}{\textbf{Fact 2}}, $\St-\mu^{\G}$ is a $(\wP,\G)$-martingale on $\id{1}{T-1}$, and that $\mu_t^{\G}$ is obtained by replacing $\Xm$ in \eqref{drift} by $\St$, i.e.,
\begin{equation}\label{drift_Y}
\mu_t^{\G}:=\sum_{s=1}^t\frac{\textless \St_t,\widehat\q^{\cm}\textgreater_s^{\wP}\big|_{\cm=\Gm}}{\widehat\q_{s-1}^{\Gm}}.
\end{equation}
\noindent
\underline{\textbf{Information drift and Malliavin derivative}}
In line with Imkeller's approach \cite{Imkeller_2003}, we can relate the information drift to the random variable $\Gm$ using the Malliavin derivative $\Dm$ and provide an alternative interpretation of Blanchet-Scalliet and Jeanblanc's result \cite{Blanchet_Jeanblanc_2020}, Proposition 2.3. This is our third significant result.
\begin{proposition}\label{Drift_Malliavin_Th}
The information drift $\mu^{\G}$ defined by \eqref{drift} can be written for any $t\in\id{1}{T-1}$ as
\begin{equation}\label{def_muG_eq}
\mu_t^{\G}=\sum_{\ell\in\{-1,1\}}\frac{a_{\ell}\; \mathbf E_{\wP}\big[\Dm_{(t,\ell)} \,\widehat\q_t^{\cm}|\f{t-1}\big]\big|_{\cm=\Gm}}{\widehat\q_{t-1}^{\Gm}},
\end{equation}
where for any $\ell\in\{-1,1\}$, $ a_{\ell}=\sum_{k\in\{-1,1\}}c_k\gE_{\wP}[\D \widehat\Zm_{(1,k)}\D \widehat\Rm_{(1,\ell)}]$, i.e.,
\begin{equation}\label{Drift_atkl_def_eq}
a_{1}=\frac{\lambda \widehat p(1-\lambda \widehat p)(b-r)}{1+r}+\frac{\lambda^2\widehat p(1-\widehat p)(a-r)}{1+r} \;\;\text{and}\;\;a_{-1}=\frac{\lambda\,(1-\lambda)(1-\widehat p)(1+2\lambda^2\widehat{p})(a-r)}{(1-\lambda\widehat p)(1+r)}.
\end{equation} 
\end{proposition}
\begin{remark}
This result is the discrete analogue of the formula (17) in Imkeller \cite{Imkeller_2003}. Classical Malliavin's derivative (in the Wiener space) enjoys the chain rule, so that the formula exhibited by Imkeller elegantly reduces in the continuous case (with the corresponding notations) to $\mu_t^{\G}=\nabla_t\log(p(\cdot,\cm))|_{\cm=\Gm}$.
\end{remark}
\noindent
\underline{\textbf{An example}} $\Gm=\mathbf 1_{\{\Sm_T\in[c,d]\}}$. For the sake of simplicity, assume here that $(1+a)(1+b)=1$ and take $\Sm_0=1$.
Let $c,d$ real positive numbers such that $(1+b)^{-T}\leqslant c<d \leqslant (1+b)^T$. Consider the case $\Gm=\mathbf 1_{\{\Sm_T\in[c,d]\}}$, i.e., the insider knows whether the terminal price of the asset is between $c$ and $d$.
\begin{itemize}
\item For any $t\in\id{1}{T}$, let $\chi_t^{\pm}=\sum_{s=1}^t\car_{\{\eta(s,\pm 1)=1\}}$, i.e., $\chi_t^{+}$ (resp. $\chi_t^{-}$) is the $\F_t$-measurable random variable that indicates the number of jumps with mark $1$ (resp. $-1$) until $t$.
\item For any $t\in\id{1}{T}$, $y$, let $\mathfrak n_{t,y}^{+}=\max\{n\in\id{1}{t}\,:\, (1+b)^{n}(1+r)^{t-n}\leqslant y\}$ and $\mathfrak n_{t,y}^{-}=\max\{n\in\id{1}{t}\,:\,(1+b)^{-n}(1+r)^{t-n}\geqslant y\}$
\item For any $t\in\id{1}{T}$, let us define $F$ defined from the cumulative function of $\Sm_t$ by
\begin{align}\label{Eq_def_CDF}
F(t,x,y)
\nonumber
&:=\gP(\{x\leqslant \Sm_t \leqslant y\})\\
&=\sum_{k=0}^{\mathfrak n_{t,y}^+}\sum_{\ell=0}^{\mathfrak n_{t,x}^{-}\wedge (t-k)}\binom{\mathfrak n_{t,y}^+}{k}\binom{\mathfrak n_{t,x}^{-}\wedge (t-k)}{\ell}(\lambda p)^{k}(\lambda(1-p))^{\ell}(1-\lambda)^{t-k-\ell}
\end{align}
\end{itemize}
For any $t\in\id{1}{T}$, we have
\begin{equation*}
\Sm_t=(1+b)^{\chi_t^{+}-\chi_t^{-}}(1+r)^{t-(\chi_t^{+}+\chi_t^{-})}.
\end{equation*}
\begin{proposition}\label{Prop_drift_example}
In the case where $\Gm=\mathbf 1_{\{\Sm_T\in[c,d]\}}$, the drift of information writes 
\begin{equation}\label{def_muG_eq_2}
\mu_t^{\G}=\sum_{\ell\in\{-1,1\}}\frac{a_{\ell}\; \mathbf E_{\wP}\big[\Dm_{(t,\ell)} \,\widehat\q_t^{\cm}|\f{t-1}\big]\big|_{\cm=\Gm}}{\widehat\q_{t-1}^{\Gm}},
\end{equation}
where the $ a_{\ell}$ are given by \eqref{Drift_atkl_def_eq}, and we have
\begin{equation*}
\mathbf E_{\wP}\big[\Dm_{(t,\ell)} \,\widehat\q_t^{\cm}|\f{t-1}\big]\big|_{\cm=\Gm}= \mathbf E_{\wP}\big[\Dm_{(t,\ell)} \,\widehat\q_t^{1}|\f{t-1}\big]\car_{\{\Gm=1\}}++ \mathbf E_{\wP}\big[\Dm_{(t,\ell)} \,\widehat\q_t^{-1}|\f{t-1}\big]\car_{\{\Gm=0\}}
\end{equation*}
with 
\begin{multline}\label{Eq_dirft_ex1}
\mathbf E_{\wP}\big[\Dm_{(t,1)} \,\widehat\q_t^{1}|\f{t-1}\big]
=\frac{p}{\widehat p}\frac{F\big(T-(t-1),c/[(1+b)\Sm_{t-1}],d/[(1+b)\Sm_{t-1}]\big)}{\,\widehat\Lm_{t-1}}\\-\frac{F\big(T-(t-1),c/[(1+r)\Sm_{t-1}],d/[(1+r)\Sm_{t-1}]\big)}{\widehat\Lm_{t-1}},
\end{multline}
and,
\begin{multline}\label{Eq_dirft_ex2}
\mathbf E_{\wP}\big[\Dm_{(t,-1)} \,\widehat\q_t^{1}|\f{t-1}\big]=\frac{1-p}{1-\widehat p}\frac{F\big(T-(t-1),c/[(1+a)\Sm_{t-1}],d/[(1+a)\Sm_{t-1}]\big)}{\,\widehat\Lm_{t-1}}\\
-\frac{F\big(T-(t-1),c/[(1+r)\Sm_{t-1}],d/[(1+r)\Sm_{t-1}]\big)}{\widehat\Lm_{t-1}}.
\end{multline}
The quantities $\mathbf E_{\wP}\big[\Dm_{(t,\pm1)} \,\widehat\q_t^{-1}|\f{t-1}\big]$ are obtained by replacing in \eqref{Eq_dirft_ex1} and \eqref{Eq_dirft_ex2} $F$ by $1-F$. 
\end{proposition}
\noindent
The numerators in \eqref{Eq_dirft_ex1} and  \eqref{Eq_dirft_ex2} can be viewed as the "partial derivatives" of $F$ with respect to a jump at time $t$ and directions $\pm 1$.

\subsection{Insider's optimal hedging strategy}\label{Subsec_insider_hedging}
For the sake of simplicity, \textbf{we assume again in this subsection that $r=0$} and we work directly with discounted prices. \\

In this subsection, we address the question: can we find what we refer to as an \textit{optimal hedging strategy} replicating some claim $\Fm$? In other words, given a $\G_{T-1}$-measurable bounded contingent claim $\Fm$, we are seeking a strategy $\psi\in\sS_{\HH}(x)$ that satisfies $\Vm_{T-1}(\psi)=\Fm$ and is independent of the choice of the \textit{insider's optimal probability measure} in $\sP^{\mathrm{opt},\G}$ (defined by \eqref{Eq_insider_opt_measure_set}).\\ 
Our result is based on the application of Clark's formula in $(\Omega,\cA,(\G_t)_{t\in\id{1}{T-1}},\wQ)$. This writes, for any $\G_{T-1}$-measurable random variable $\Fm$,
\begin{equation}\label{Clark_wQformula_eq}
\Fm=\gE_{\wQ}[\Fm|\G_0]+\sum_{t\in\id{1}{T-1}}\sum_{k\in\{-1,1\}}\gE_{\wQ}\big[\Dm_{(t,k)}\Fm\,|\,\G_{t-1}\big]\D\widehat\Rm_{(t,k)}.
\end{equation}
We can then deduce our fourth and last significant result, in fact the most important one of this section.
\begin{theorem}[Optimal hedging formula]\label{Hedging_agent_th}
Every $\G_{T-1}$-measurable claim $\Fm$ is reachable for the insider. 
The $\G$-strategy $\psi=(\alpha_t,\varphi_t)_{t\in\id{1}{T-1}}$
defined on the one hand by $\varphi_0=0$,
\begin{equation*}\label{Hedging_agent_eq}
\vp_t=(1+r)^{t-T+1}\frac{\gE_{\wQ}\big[\Dm_{(t,-1)}\Fm\,|\,\g{t-1}\big]}{(a-r)\Sm_{t-1}},
\end{equation*}
and, on the other one, by $\alpha_0=(1+r)^{-T+1}\gE_{\wQ}[\Fm|\G_0]=(1+r)^{-T+1}\gE[\Fm]$ and for any $t\in\id{1}{T-1}$,
\begin{equation*}
\alpha_t=\alpha_{t-1}-\frac{(\varphi_t-\varphi_{t-1})\Sm_{t-1}}{\A_{t-1}},
\end{equation*}
is a $\G$-predictable self-financing strategy that replicates $\Fm$. \\
Moreover, the strategy is independent of the choice of the \textit{optimal} probability measure $\wQ\in\sP^{\mathrm{opt},\G}$ for $\Um\in\mathscr U^{\mathfrak b,\G}$.
\end{theorem}
\begin{remark}
Consider the case $\Gm=\mathbf 1_{\{\Sm_T\in[c,d]\}}$. By applying the hedging formula to $\Fm=\widehat{\Vm}_{t}^{\G,\mathbf{log}}$ ($t\in\id{1}{T-1}$) given by \eqref{Eq_opt_insider_log}, we get insider's optimal strategy until time $t$. In the same vein as for the drift of information (see \eqref{Eq_dirft_ex1} and  \eqref{Eq_dirft_ex2}), we can prove that $(\varphi_s)_{s\in\id{1}{t}}$ writes in terms of the function $F$ defined by \eqref{Eq_def_CDF} and its "partial derivatives" with respect to a jump at time $s$ in the directions $\pm1$: heuristically, it seems that the insider uses her additional information to learn about the directions of the underlying jump process and adjusts her portfolio accordingly. 
\end{remark}

\section{Conclusion and perspectives}
In this paper, we have explored various aspects of insider trading in a jump-binomial model of the financial market. This constitutes a discrete-time incomplete market model and emerges as a novel representation of the classical trinomial market model as a volatility model. It is based on a marked binomial process that acts as the sequence of i.i.d. random variables underlying the original trinomial model. All the results were established at a lower cost by using the volatility structure of the jump-binomial model and the stochastic analysis tools provided for marked binomial processes in \cite{Halconruy_2021_binomial}.\\

We presented our results according to the two perspectives we addressed successively.
First, having in mind to quantify the benefit that the insider gains from using the additional information available to her, we provided new explicit formulas for the expected additional utility (logarithmic, exponential, power) compared to an ordinary agent. We interpreted the measure of the benefit obtained in the context of information theory, connecting it to the entropy of the additional information.
Second, we investigated the impact of considering the insider's information level instead of the ordinary agent's. Specifically, we provided a novel interpretation of the information drift that characterizes the preservation of martingales under a change in filtration. Additionally, we explicitly computed the optimal hedging strategy for the insider. Both results rely on a recent version of the Malliavin calculus developed for marked binomial processes in \cite{Halconruy_2021_binomial}.\\

Two other intriguing paths of exploration emerge to extend the trajectory of this paper.
A first and natural question would be to rule on \textit{arbitrage opportunities} for the insider, i.e., to investigate whether the additional information at the insider's disposal allows her to make profits without taking risks. This fundamental question, which remains unexplored in this paper, has been examined in prior research. In a discrete-time setting, the reader can refer, to the works of Choulli and Deng \cite{Choulli_Deng_2018} in a progressive enlargement setting, to the work of Blanchet-Scalliet, Hillairet and Jiao \cite{Blanchet_Hillairet_Jiao} for a \textit{successive} enrichment by a family of enlargement of filtrations, and to the works of Burzoni, Frittelli and Maggis in \cite{Burzoni_Frittelli_Maggis_2017}, in the frame of \textit{uncertainty} models (without a unique probability reference measure). 
\\ 
On the other hand, all our results hold under the assumption of minimal impact from the insider's trading decisions on price evolution. An interesting extension would be to explore market models where the insider's actions directly influence the agent's decision-making process. We could imagine letting the dynamics of the risky asset depend on the insider's strategy, in the same spirit as Kohatsu-Higa and Sulem \cite{Kohatsu_Sulem}. This investigation could also raise questions about the existence of potential partial equilibrium, as defined by Hata and Kohatsu-Higa \cite{Hata_Kohatsu}.

\appendix

\section{Proofs}
\begin{proof}[\textit{Proof of Proposition \ref{Opt_port_ag_ter_T_prop}}] As described in our procedure, let us start by solving the optimization problem in the one-period case $T=1$. We use Theorem 3.1.3 of Delbaen, Schachermayer \cite{Delbaen_Schachermayer}. Consider the dual problem  \eqref{dual_optTerAg_eq}  associated to \eqref{optTerAg_eq} in one-period case ($T=1$), i.e.,
\begin{equation*}
\Psi_T^{\Bi,u}(y)=\gE_{\wPb}\Big[v^{u}\Big(y\frac{d\wPb}{d\gP}\Big)\Big],
\end{equation*}
where $v^{u}$ denotes the conjugate function of $u$.
Recall and translate the results of Delbaen and Schachermayer \cite{Delbaen_Schachermayer}, Theorem 3.1.3 into our frame in the case by taking $\wPb$ as the unique martingale measure equivalent to $\gP$. The solution of $\Phi_T^{\Bi,u}(x)$ is the portfolio whose discounted value $\widehat{\Vm}_{T}^{\F,u}$ is
\begin{equation}\label{Del_Scha_general_opt_formula}
\widehat{\Vm}_{T}^{\F,u}=I\left(y\frac{d\wPb}{d\gPb}\right)
\end{equation}
where $I$ is the function $I=-(v^{u})'$. Moreover,  $x$ and $y$ are related via the relations $(\Phi_T^{\Bi,u})'(x)=y$ or equivalently $x=-(\Psi_T^{\Bi,u})'(y)$.
 Then, it is enough to compute $(\Psi_1^{\Bi,u})'(y)$ to get an explicit relation between $x$ and $y$ and to deduce $\widehat \Vm_{T}^{\F,u}$. \\
Denote $ p_{0}=\gPb(\{\eta(T,\cdot)=0\})$, $\widehat p_{0}=\wPb(\{\eta(T,\cdot)=0\})$ and for $k\in\{-1,1\}$, $p_{k}=\gPb(\{\eta(T,k)=1\})$, $\widehat p_{k}=\wPb(\{\eta(T,k)=1\})$.
\vspace{3pt}\\
\textbf{\underline{Logarithmic utility}}: For $y\in\R_+^*$,
\begin{align*}
\Psi_T^{\Bi,\log}(y)
&=-1-\sum_{k\in \{-1,0,1\}}p_k\log\Big(y\frac{\widehat{p}_k}{p_k}\Big)\\
&=-1-\log(y)+\mathfrak D_{\F_T}(\gPb||\wPb).
\end{align*}
Then $x=-\Psi_T'(y)=1/y$ and the result follows by using \eqref{Del_Scha_general_opt_formula} with $I(\cdot)=1/(\cdot)$.\\ 
\textbf{\underline{Exponential utility}}: For $y\in\R_+^*$, 
\begin{align*}
\Psi_T^{\Bi,\exp}(y)
&=\sum_{k\in \{-1,0,1\}}y\frac{\widehat{p}_k}{p_k}\Big(\log\Big(y\frac{\widehat{p}_k}{p_k}\Big)-1\Big)p_k\\
&=y\log(y)-y+y\mathfrak D_{\F_T}(\wPb||\gPb).
\end{align*}
Then $x=-\Psi_T'(y)=-\log(y)-\mathfrak D_{\F_T}(\wPb||\gPb)$ and then $y=\exp(-x-\mathfrak D_{\F_T}(\wPb||\gPb))$. The result follows by using \eqref{Del_Scha_general_opt_formula} with $I=-\log$.\\
\textbf{\underline{Power utility}}: For $y\in\R_+^*$,
\begin{align*}
\Psi_T(y)
&=-\frac{1}{\beta}\sum_{k\in \{-1,0,1\}}p_ky^{\beta}\Big(\frac{\widehat{p}_k}{p_k}\Big)^{\beta}\\
&=-\frac{y^{\beta}}{\beta}\gE\big[\widehat \Lm^{\beta}\big]
\end{align*}
where $\widehat \Lm=d\wPb/d\gPb$.
Then $x=-\Psi_T'(y)=y^{\beta-1}\gE\big[\widehat \Lm^{\beta}\big]$ and the result follows by using \eqref{Del_Scha_general_opt_formula} with $I(y)=y^{\beta-1}$. \\
As described in the procedure, (still considering $T=1$) we deduce the optimal discounted value of the portfolio in the jump-binomial model by replacing everywhere needed $\gPb$ and $\wPb$ respectively by $\gP$ and $\wP$. \\
To extend the result to the multi-period case ($T\geqslant 2$), we define for all $s,t\in\id{0}{T}$ such that $s<t$, $\overline{\F}^{s,t}=(\F_r)_{s+1\pp r\pp t}$. The expression of $\Phi_t^{\F,u}(x)$ can be deduced from the identity $\Theta_{t}^{\F,u}(x)=\Phi_{T-t}^{\F,u}(x)$, together with the solution of the following induction system
\begin{displaymath}\label{Theta_def_eq}
\left\{
\begin{array}{rcl}
\Theta_T^{\F,u}(x)&=&u(x)\vspace{2pt}\\
\Theta_{t-1}^{\F,u}(x)&=&\underset{\psi\,\in\sS_{\overline{\F}^{t-1,t}}(x)}\sup\,\esp{\Theta_t^{\F}(x+\varphi_t\D \St_{t})}\; ; \; t\in\id{1}{T},
\end{array}
\right.
\end{displaymath}
where the supremum is taken over the strategies $\psi=(\alpha,\varphi)\in\sS_{\overline{\F}^{t-1,t}}(x)$ and $\varphi\in\R$. For $t=T-1$, since the $\D\Sm_t$ are independent,
\begin{align*}
\Theta_{T-1}^{\F,u}(x)
&=\underset{\psi\,\in\sS_{\overline{\F}^{T-1,T}}(x)}\sup\,\esp{\Theta_T^{\F}\left(x+\varphi\,\D \St_{T}\right)\,\Big|\,\f{T-1}}\\
&=\underset{\psi\,\in\sS_{\overline{\F}^{0,1}}(x)}\sup\,\gE_{\gP^{1,T}}\big[{\log( x+\varphi\,\D \St_{1})}\big],
\end{align*}
where $\gP^{1,T}$ is the probability measure on $(\Omega,\F_1)$ such that $\gP^{1,T}(\{\eta(1,\cdot)=0\})=\gP(\{\eta(T,\cdot)=0\})$ and $\gP^{1,T}(\{\eta(1,k)=1\})=\gP(\{\eta(T,k)=1\})$ for $k\in\{-1,1\}$. For any $s\in\id{1}{T-1}$, we get $\Theta_{s}^{\F,u}(x)$ by downward induction. Last, we obtain $\Phi_{t}^{\F,u}(x)=\Theta_{s}^{\F,u}(x)$ by letting $t=T-s$.
Hence the result.
\end{proof}
\begin{proof}[Proof of Lemma \ref{Lem_DetltaS}]
Let us first note that for any $t\in\id{0}{T}$, the $\sigma$-algebra $\G_{t}$ is generated by the set
\begin{equation*}
\{\B\cap\{\Gm\in \Cm\}\; ; \; \B\in\F_{t},\, \Cm\in\mathscr G\}.
\end{equation*}
By \textcolor{teal}{\textbf{Fact 4}}, we have for $t\in\id{0}{T-1}$,
\begin{equation}\label{Eq_wQ_prop_independence}
\wQ(\B_t\cap \{\Gm\in \Cm\})=\wQ(\B_t)\wQ(\{\Gm\in \Cm\})=\wP(\B_t)\gP(\{\Gm\in \Cm\}).
\end{equation}
For any $t\in\id{1}{T-1}$, let $\A_{t-1}=\B_{t-1}\cap\{\Gm\in\Cm\}\in\G_{t-1}$ where $\B_{t-1}\in\F_{t-1}$ and $\Cm\in\sG$. Assume that $\gP(\A_{t-1})>0$. Since $1/\widehat\q_t^{\Gm}$ is positive for all $t\in\id{0}{T-1}$, we have $\wQ(\A_{t-1})>0$.  For any $t\in\id{1}{T-1}$,
\begin{align*}
\wQ(\A_{t-1})\wQ(\{\eta(t,1)=1\}|\A_{t-1})
&=\gE_{\wQ}\big[\car_{\{\eta(t,1)=1\}}\car_{\B_{t-1}}\car_{\{\Gm\in\Cm\}}\big]\\
&=\wQ(\{\eta(t,1)=1\}\cap\B_{t-1})\wQ(\{\Gm\in\Cm\})\\
&=\wP(\{\eta(t,1)=1\})\wQ(\B_{t-1})\wQ(\{\Gm\in\Cm\})\\
&=\wQ(\{\eta(t,1)=1\})\wQ(\A_{t-1})=\wP(\{\eta(t,1)=1\})\wQ(\A_{t-1})
\end{align*}
where we used \eqref{Eq_wQ_prop_independence} and that $\eta(t,1)$ is independent of $\F_{t-1}$. The penultimate equality means that $\{\eta(t,k)=1\}$ and $\G_{t-1}$ are independent under $\wQ$. Since $\G_{t-1}$ is generated by the set $\{\B\cap \Cm \; ; \; \B\in\F_{t-1},\, \Cm\in\mathscr G\}$, the property extends to $\G_{t-1}$ and holds $\gP$-almost surely (we have only considered sets $\A_{t-1}$ whose $\gP$-measure is non-zero) via monotone class theorem. This provides the first part of the statement as $\wP(\{\eta(t,1)=1\})=\lambda\widehat p$.
A similar statement can be obtained in the case $k=-1$. As a consequence, for any $t\in\id{1}{T-1}$,
\begin{align*}
\frac{\D\St_t}{\St_{t-1}}
\nonumber
&=\frac{b-r}{1+r}\car_{\{\eta(t,1)=1\}}+\frac{a-r}{1+r}\car_{\{\eta(t,-1)=1\}}\\
\nonumber
&=\frac{b-r}{1+r}\Big[\car_{\{\eta(t,1)=1\}}-\wQ(\{\eta(t,1)=1\}|\G_{t-1})\Big]+\frac{a-r}{1+r}\Big[\car_{\{\eta(t,-1)=1\}}\\
&\nonumber-\wQ(\{\eta(t,-1)=1\}|\G_{t-1})\big]
+\bigg[\frac{b-r}{1+r}\wQ(\{\eta(t,1)=1\}|\G_{t-1})+\frac{a-r}{1+r}\wQ(\{\eta(t,-1)=1\}|\G_{t-1})\bigg]\\
&=\frac{b-r}{1+r}\big[\car_{\{\eta(t,1)=1\}}-\lambda\widehat p\big]+\frac{a-r}{1+r}\big[\car_{\{\eta(t,-1)=1\}}-\lambda(1-\widehat p)\big]=\frac{b-r}{1+r}\Delta\widehat\Zm_{(t,1)}+\frac{a-r}{1+r}\Delta\widehat\Zm_{(t,-1)}.
\end{align*}
 Hence the result.
\end{proof}
\begin{proof}[Proof of Theorem \ref{Opt_port_ins_ter_T_prop}] 
As for the ordinary agent, we can deduce insider's maximum expected utility by solving the associated dual problem in the binomial model. In this one, insider's optimal portfolios are simply obtained (at time $t\in\id{1}{T-1}$) from agent's by replacing everywhere needed $\F_t$, $\wPb$  respectively by $\G_t$ and $\widehat\Q^{\Bi,u}$ (for $u\in\{\mathbf{log},\mathbf{exp},\mathbf{pow}\}$) with
\begin{equation*}
\widehat\Q^{\Bi,\log}:=\Um^{\mathbf{log}}*\wQ^{\Bi},\quad \widehat\Q^{\Bi,\mathbf{exp}}:=\Um^{\mathbf{exp}}*\wQ^{\Bi} \quad\text{and}\quad \widehat\Q^{\Bi,\mathbf{pow}}:=\Um^{\mathbf{pow}}*\wQ^{\Bi},
\end{equation*}
where $\Um^{\mathbf{log}},\Um^{\mathbf{exp}}$ and $\Um^{\mathbf{pow}}$ are defined by \eqref{U_opt_ins_def_eq}.
We get then:\vspace{2pt}\\
\noindent
\underline{\textbf{Logarithmic utility}}: 
$\widehat{\Vm}_{t}^{\Bi,\G,\mathbf{log}}=x\cdot\dfrac{d\gPb}{d\wQ^{\Bi,\mathbf{log}}}\bigg|_{\G_t}$.
\vspace{2pt}\\
\noindent
\underline{\textbf{Exponential utility}}: $\widehat{\Vm}_{t}^{\Bi,\G,\mathbf{exp}}=x+\mathfrak D_{\G_t}(\widehat\Q^{\Bi,\mathbf{exp}}||\gPb)+\log\bigg(\dfrac{d\gPb}{d\widehat\Q^{\Bi,\mathbf{exp}}}\Big|_{\G_t}\bigg)$.
\vspace{2pt}\\
\noindent 
\underline{\textbf{Power utility}}: $\widehat{\Vm}_{t}^{\Bi,\G,\mathbf{pow}}=x\cdot\gE\Big[\Big(\dfrac{d\widehat\Q^{\Bi,\mathbf{pow}}}{d\gPb}\Big|_{\G_t}\Big)^{\beta}\Big]^{-1}\cdot\bigg(\dfrac{d\widehat\Q^{\Bi,\mathbf{pow}}}{d\gPb}\Big|_{\G_t}\bigg)^{\beta-1},$ where $\beta=\alpha/(\alpha-1)$.\vspace{2pt}\\
Then, the results can be stated in the jump-binomial model by replacing everywhere needed $\gPb$ and  $\widehat\Q^{\Bi,u}$ respectively by $\gP$ and $\widehat\Q^{u}$ with
\begin{equation*}
\widehat\Q^{u}=(1-\lambda)\gPc+\lambda \widehat\Q^{\mathfrak b,u}=(1-\lambda)\gPc+\lambda(\Um^{u}*\wQb).
\end{equation*}
Hence the result.
\end{proof}
\begin{proof}[Proof of Theorem \ref{additional_utility_th}]
The following explicit expressions for insider's $u$-additional expected utility are deduced from Theorem \ref{Opt_port_ins_ter_T_prop}.\\
\textbf{\underline{Logarithmic utility}}: For $t\in\id{1}{T-1}$,
\begin{align*}
\Phi_t^{\G,\mathbf{log}}(x)-\Phi_t^{\F,\mathbf{log}}(x)
&=\gE\bigg[\log\Big(x\cdot\frac{d\gP}{d\wQ^{\mathbf{log}}}\bigg|_{\G_t}\Big)-\log\Big(x\cdot\frac{d\gP}{d\wP}\bigg|_{\F_t}\Big)\bigg]\\
&=\gE\bigg[\log\Big(\frac{d\wP}{d\wQ^{\mathbf{log}}}\bigg|_{\G_t}\Big)\bigg]\\
&=\gE\bigg[\log\Big(\frac{d\wPb}{d\wQ^{\Bi,\mathbf{log}}}\bigg|_{\G_t}\Big)\bigg]\\
&=\gE_{\gPb}\bigg[\log\Big(\frac{1}{\Um^{\mathbf{log}}}\Big)\bigg]+\gE_{\gPb}\Big[\log\big(\q_t^{\Bi,\Gm}\big)\Big],
\end{align*}
where the third line comes from the definitions $\wP=(1-\lambda)\gP^{\mathfrak c}+\lambda \gPb$ and $\widehat\Q^{\mathbf{log}}=(1-\lambda)\gP^{\mathfrak c}+\lambda\Um^{\mathbf{log}}*\wQb$.
Moreover, by definition of $\q^{\Bi,\Gm}$ and $\q^{\Gm}$ and since $\Gamma$ is finite,
\begin{align*}
\gE_{\gPb}\big[\log(\q^{\Bi,\Gm}_t)\big]
&=\esp{\log(\q^{\Gm}_t)}\\
&=\esp{\sum_{\cm\in\Gamma}\log(\q_t^{\cm})\,\gP(\{\Gm=\cm\}\,|\,\F_t)}\\
&=\esp{\sum_{\cm\in\Gamma}\log\big(\gP(\{\Gm=\cm\}\,|\,\F_t)\big)\,\gP(\{\Gm=\cm\}\,|\,\F_t)}\\
&\hspace{50pt}-\sum_{\cm\in \Gamma}\log\big(\gP(\{\Gm=\cm\})\big)\esp{\esp{\car_{\{\Gm=\cm\}}\,|\,\F_t}}\\
&=\esp{\sum_{\cm\in \Gamma}\log\big(\gP(\{\Gm=\cm\}\,|\,\F_t)\big)\gP(\{\Gm=\cm\}|\F_t)}\\
&\hspace{150pt}-\sum_{\cm\in \Gamma}\log\big(\gP(\{\Gm=\cm\})\big)\,\gP(\{\Gm=\cm\})\\
&=\mathrm{Ent}(\Gm)-\mathrm{Ent}(\Gm\,|\,\F_t),
\end{align*} 
where we get the second equality by conditioning on $\F_t$.  \\
\textbf{\underline{Exponential utility}}: For $t\in\id{1}{T-1}$,
\begin{align*}
\Phi_t^{\G,\mathbf{exp}}(x)-\Phi_t^{\F,\mathbf{exp}}(x)
&=-\exp(-x-\mathfrak D_{\G_t}(\widehat\Q||\gP))\cdot\gE\bigg[\frac{d\widehat\Q^{\mathbf{exp}}}{d\gP}\Big|_{\G_t}\bigg]\\
&\hspace{120pt}+\exp(-x-\mathfrak D_{\F_t}(\widehat\gP||\gP))\cdot\gE\bigg[\frac{d\widehat\gP}{d\gP}\Big|_{\F_t}\bigg]\\
&=-\exp(-x-\mathfrak D_{\G_t}(\widehat\Q||\gP))\cdot\gE\bigg[\frac{d\widehat\Q^{\mathbf{exp}}}{d\gP}\Big|_{\G_t}\bigg]\\
&\hspace{190pt}+\exp(-x-\mathfrak D_{\F_t}(\widehat\gP||\gP)),
\end{align*}
where, since $\wP$ is equivalent to $\gP$, $\gE[(d\gP/d\wP)|_{\F_t}]=1$. 
Using that $\gP=(1-\lambda)\gP^{\mathfrak c}+\lambda\gPb$ and $\widehat\Q=(1-\lambda)\gP^{\mathfrak c}+\lambda \Um^{\mathbf{exp}}*\wQ^{\Bi}$ where $ \Um^{\mathbf{exp}}$ satisfies $\gE_{\wQb}\big[\Um^{\mathbf{exp}}\big]=1$, we can check that
\begin{equation*}
\gE\bigg[\frac{d\widehat\Q^{\mathbf{exp}}}{d\gP}\Big|_{\G_t}\bigg]=1.
\end{equation*}
\\
\textbf{\underline{Power utility}}: For $t\in\id{1}{T-1}$, since $\alpha(\beta-1)=\beta$, 
\begin{align*}
\Phi_t^{\G,\pow}(x)-\Phi_t^{\F,\pow}(x)
&=\frac{x^{\alpha}}{\alpha}\gE\Big[\Big(\frac{d\widehat\Q^{\pow}}{d\gP}\Big|_{\G_t}\Big)^{\beta}\Big]^{-\alpha}\cdot\gE_{\gP}\Big[\Big(\frac{d\widehat\Q^{\pow}}{d\gP}\Big|_{\G_t}\Big)^{\beta}\Big]\\
&\hspace{100pt}-\frac{x^{\alpha}}{\alpha}\gE\Big[\Big(\frac{d\wP}{d\gP}\Big|_{\F_t}\Big)^{\beta}\Big]^{-\alpha}\cdot\gE\Big[\Big(\frac{d\wP}{d\gP}\Big|_{\F_t}\Big)^{\beta}\Big]\\
&=\frac{x^{\alpha}}{\alpha}\gE\bigg[\Big(\frac{d\widehat\Q^{\pow}}{d\gP}\Big|_{\G_t}\Big)^{\beta}\bigg]^{1-\alpha}-\frac{x^{\alpha}}{\alpha}\gE\Big[\big(\widehat\Lm_t\big)^{\beta}\Big]^{1-\alpha}.
\end{align*}
Hence the result.
\end{proof}
\begin{proof}[Proof of Corollary \ref{additional_utility_cor}] Let $t\in\id{1}{T-1}$. We deduce the following bounds from Theorem \ref{additional_utility_th}.
\\
\textbf{\underline{Exponential utility}}: There exists $\kappa\in(0,1)$ such that
\begin{align*}
\mathcal U_t^{\mathbf{exp}}(x)
&=\lambda\exp\big(-x\big[(1-\kappa)\mathfrak D_{\G_t}(\widehat\Q||\gP)+\kappa \mathfrak D_{\F_t}(\widehat\gP||\gP)\big]\big)\big(\mathfrak D_{\F_t}(\widehat\gP||\gP)-\mathfrak D_{\G_t}(\widehat\Q||\gP)\big)\\
&=\lambda\exp\big(-x\big[(1-\kappa)\mathfrak D_{\G_t}(\widehat\Q||\gP)+\kappa \mathfrak D_{\G_t}(\widehat\gP||\gP)\big]\big)\mathfrak D_{\G_t}(\widehat\gP||\widehat\Q)\\
&\leqslant\lambda\exp\big(-x\mathfrak D_{\G_t}((1-\kappa)\widehat\Q+\kappa\widehat\gP||\gP)\mathfrak D_{\G_t}(\wP||\wQ)\\
&
=\exp\big(-x\mathfrak D_{\G_t}(\mathbf{M}_\kappa||\gP)\big)\big[\mathrm{Ent}(\Gm)-\mathrm{Ent}(\Gm\,|\,\F_t)-\gE_{\gPb}[\log(\Um^{\mathbf{log}})]\big]
\end{align*}
where we set $\mathbf{M}_\kappa:=(1-\kappa)\widehat\Q+\kappa\widehat\gP$ and we have used that the map $(\wP,\gP)\mapsto \mathfrak D_{\G_t}(\wP||\gP)$ is jointly convex.\\
\textbf{\underline{Power utility}} For the sake of readability, let us note $\widehat\Lm_t^{\mathbf{\pow}}=(d\widehat\Q^{\pow}/d\gP)|_{\G_t}$.
There exists $\kappa\in(0,1)$ such that
\begin{align*}
\mathcal U_t^{\mathbf{pow}}(x)
&\leqslant\frac{x^{\alpha}}{\alpha}\Big(\gE\big[\big(\widehat\Lm_t^{\mathbf{\pow}}\big)^{\beta}\big]-\gE_{\gP}\big[\big(\widehat\Lm_t\big)^{\beta}\big]\Big)^{1-\alpha}\\
&= \frac{x^{\alpha}}{\alpha}\Big(\gE\big[\big(\exp(\beta\log(\widehat\Lm_t^{\mathbf{\pow}})\big)\big]-\gE\big[\exp\big(\beta\log(\widehat\Lm_t)\big)\big]\Big)^{1-\alpha}\\
&\leqslant\frac{|\beta|^{1-\alpha}x^{\alpha}}{\alpha}\Big\|\exp\big(\beta\log(\widehat\Lm_t^{\mathbf{pow},\kappa})\big)\Big\|_{\infty}^{1-\alpha}\Big(\gE\big[\log(\widehat\Lm_t)\big]-\gE\big[\log(\widehat\Lm_t^{\mathbf{pow}})\big]\Big)^{1-\alpha}\\
&=\frac{|\beta|^{1-\alpha}x^{\alpha}}{\alpha}\Big\|\big(\widehat\Lm_t^{\mathbf{pow},\kappa}\big)^{\beta}\Big\|_{\infty}^{1-\alpha}\big[\mathfrak D_{\G_t}(\wP||\wQ^{\mathbf{pow}})\big]^{1-\alpha}\\
&=\frac{|\beta|^{1-\alpha} x^{\alpha}}{\alpha}\Big\|\big(\widehat\Lm_t^{\mathbf{pow},\kappa}\big)^{\beta}\Big\|_{\infty}^{1-\alpha}\big[\mathrm{Ent}(\Gm)-\mathrm{Ent}(\Gm\,|\,\F_t)-\gE_{\gPb}[\log(\Um^{\mathbf{pow}})]\big]^{1-\alpha}
\end{align*}
where we have defined $\widehat\Lm_t^{\mathbf{pow},\kappa}$ as the random variable such that $\log(\widehat\Lm_t^{\mathbf{pow},\kappa}):=(1-\kappa)\log(\widehat\Lm_t)+\kappa\log(\widehat\Lm_t^{\mathbf{pow}})$ and used that $x\in\R_+^{*}\mapsto x^{1-\alpha}$ is $(1-\alpha)$-H\"older continuous.
\end{proof}
\begin{proof}[\textit{Proof of Proposition \ref{Drift_Malliavin_Th}}]
Consider the process $\mu^{\G}$ defined by \eqref{drift} by taking $\Xm=\overline{\Ym}$.
The proof directly derives from the Clark-Ocone formula \eqref{Clark_formula_mart_eq} applied to $\widehat\q^{\cm}=\mathbf q^{\cm}/\widehat\Lm$ (with $\cm\in\Gamma$) that is a $(\wP,\F)$-martingale on $\id{1}{T-1}$. Taking $s=t-1$ provides
\begin{equation*}
\D \widehat\q_t^{\cm}=\widehat\q_{t}^{\cm}-\widehat\q_{t-1}^{\cm}=\sum_{\ell\in\{-1,1\}}\gE_{\wP}[\Dm_{(t,\ell)} \,\widehat\q_t^{\cm}\,|\,\f{t-1}]\,\D \widehat\Rm_{(t,\ell)}.
\end{equation*}
As stated in Lemma 1.4 of  Blanchet \textit{et al.} \cite{Blanchet_Jeanblanc_Romero_2019}, for two $\F$-adapted processes $\Um$ and $\Km$, and a probability measure $\gP$, $\langle \Um,\Km\rangle_0^\gP=0$ and $\D\langle \Um,\Km\rangle_t^{\gP}=\mathbf E_\gP[\D\Um_t\D\Km_t\,|\,\F_{t-1}]$ for all $t\in\id{1}{T}$, where $\langle\Um,\Km\rangle^{\gP}$ is the angle bracket, i.e., the $\F$-predictable process such that $(\Um_t\Km_t-\langle\Um,\Km\rangle_t^{\gP})_{t}$ is a $(\gP,\F)$-martingale. By \eqref{Eq_DetltaS},
\begin{equation*}
\Delta\St_t=\frac{1}{\St_{t-1}}\sum_{k\in\{-1,1\}}c_k\D \widehat\Zm_{(t,k)}
\end{equation*}
with $c_1=[b-r]/[1+r]$ and $c_{-1}=[a-r]/[1+r]$.
Then we get, for any $t\in\id{1}{T-1}, \cm\in\Gamma$,
\begin{align*}
\D\textless \St,\widehat\q^{\cm}\textgreater_t^{\wP}
&=\frac{1}{\St_{t-1}}\mathbf E_{\wP}\Bigg[\sum_{k\in\{-1,1\}}c_k\D \widehat\Zm_{(t,k)}\,\sum_{\ell\in\{-1,1\}}\mathbf E_{\wP}[\Dm_{(t,\ell)} \,\widehat\q_t^{\cm}|\f{t-1}]\,\D\widehat\Rm_{(t,\ell)}\Big|\F_{t-1}\Bigg]\\
&=\frac{1}{\St_{t-1}}\sum_{k\in\{-1,1\}}\sum_{\ell\in\{-1,1\}}c_k\mathbf E_{\wP}[\Dm_{(t,\ell)} \,\widehat\q_t^{\cm}|\f{t-1}]\mathbf E_{\wP}\big[\D\widehat\Zm_{(t,k)}\D \widehat\Rm_{(t,\ell)}\big]\\
&=\frac{1}{\St_{t-1}}\sum_{\ell\in\{-1,1\}}a_{\ell}\mathbf E_{\wP}[\Dm_{(t,\ell)} \,\widehat\q_t^{\cm}|\f{t-1}],
\end{align*}
where we define the family $\{a_{\ell},\; \ell\in\{-1,1\}\}$ by $a_{\ell}=\sum_{k\in\{-1,1\}}c_k\gE_{\wP}[\D\widehat \Zm_{(1,k)}\D\widehat \Rm_{(1,\ell)}]$. Hence the result. 
\end{proof} 
\begin{proof}[Proof of Proposition \ref{Prop_drift_example}]
For any $t\in\id{0}{T-1}$,
\begin{align*}
\gP(\{c\leqslant \Sm_T \leqslant d\}\,|\,\F_t)
&=\gP\Big(\Big\{\frac{c}{\Sm_t}\leqslant \frac{\Sm_T}{\Sm_t} \leqslant \frac{d}{\Sm_t}\Big\}\,\Big|\,\F_t\Big)\\
&=\gP\Big(\Big\{\frac{c}{k}\leqslant \Sm_{T-t} \leqslant \frac{d}{k}\Big\}\Big)\Big|_{k=\Sm_t}=F\Big(T-t,\frac{c}{\Sm_{t}},\frac{d}{\Sm_t}\Big),
\end{align*}
where we have used that the ratios $\Sm_t/\Sm_{t-1}$ are i.i.d. so that $\Sm_T/\Sm_{t}$ has the same law as  $\Sm_{T-t}$. Then, since $\Gm=\car_{\{\Sm_T\in[c,d]\}}$,
\begin{equation*}
\q_t^{1}=\frac{\gP(\{c\leqslant \Sm_T \leqslant d\}\,|\,\F_t)}{\gP(\{c\leqslant \Sm_T \leqslant d\})}=\frac{F\big(T-t,c/\Sm_{t},d/\Sm_t\big)}{F(T,c,d)}\;\;\text{and}\;\;\q_t^{0}=\frac{1-F\big(T-t,c/\Sm_{t},d/\Sm_t\big)}{1-F(T,c,d)}.
\end{equation*}
Moreover, for any $t\in\id{1}{T-1}$, as
\begin{equation*}
\widehat\Lm_t=\prod_{s=1}^{t}\bigg[\car_{\{\eta(s,\cdot)=0\}}+\frac{p}{\widehat p}\car_{\{\eta(s,1)=1\}}+\frac{1-p}{1-\widehat p}\car_{\{\eta(s,-1)=1\}}\bigg],
\end{equation*}
we have
\begin{align*}
\Dm_{(t,1)} \,\widehat\q_t^{1}
=\widehat\q_t^{1}(\pi_t(\eta)+\delta_{(t,1)})-\widehat\q_t^{1}(\pi_t(\eta))&=\frac{\q_t^{1}(\pi_t(\eta)+\delta_{(t,1)})}{\widehat\Lm_t(\pi_t(\eta)+\delta_{(t,1)})}-\frac{\q_t^{1}(\pi_t(\eta))}{\widehat\Lm_t(\pi_t(\eta))}\\
&=\frac{p}{\widehat p}\frac{F\big(T-(t-1),c/[(1+b)\Sm_{t-1}],d/[(1+b)\Sm_{t-1}]\big)}{\,\widehat\Lm_{t-1}}\\
&\quad-\frac{F\big(T-(t-1),c/[(1+r)\Sm_{t-1}],d/[(1+r)\Sm_{t-1}]\big)}{\widehat\Lm_{t-1}}.
\end{align*}
Similarly,
\begin{multline*}
\Dm_{(t,-1)} \,\widehat\q_t^{-1}=\frac{1-p}{1-\widehat p}\frac{F\big(T-(t-1),c/[(1+a)\Sm_{t-1}],d/[(1+a)\Sm_{t-1}]\big)}{\,\widehat\Lm_{t-1}}\\
-\frac{F\big(T-(t-1),c/[(1+r)\Sm_{t-1}],d/[(1+r)\Sm_{t-1}]\big)}{\widehat\Lm_{t-1}}.
\end{multline*}
Hence the result. 
\end{proof}

\begin{proof}[Proof \textit{of Theorem \ref{Hedging_agent_th}}]
Consider a $\G_{T-1}$-measurable random variable $\Fm$. \\
\underline{\textbf{Step 1: Identification of the hedging strategy}}
As a reminder, the strategy $\psi=(\alpha,\varphi)$ is self-financing if and only if the condition \eqref{Self_financing_cond} is satisfied for all $t\in\id{1}{T-1}$ so that $\Vm_{t-1}(\psi)=\alpha_{t}\A_{t-1}+\vp_{t}\Sm_{t-1}$. Let $\varphi_0=0$. Assume the existence of a $\G$-admissible strategy $\psi$ such that $\Vm_0(\psi)=x$ and which final value satisfies 
\begin{equation*}
\Vm_{T-1}(\psi)=\alpha_{T-1}\,\A_{T-1}+\vp_{T-1}\,\Sm_{T-1}=\Fm.
\end{equation*}
\underline{Step 1} As a reminder by \eqref{Eq_DetltaS}, for $t\in\id{1}{T-1}$,
\begin{equation*}
\D\St_t=\St_{t-1}\bigg[\frac{b-r}{1+r}\D\widehat\Zm_{(t,1)}+\frac{a-r}{1+r}\D\widehat\Zm_{(t,-1)}\bigg].
\end{equation*}
%
Let $\pi$ be the $\G$-predictable process such that $\pi_t=\dfrac{\vp_{t}\,\St_{t-1}}{\Vt_{t-1}(\psi)}$ for any $t\in\id{1}{T-1}$. Thus,
\begin{align*}
\Delta \overline\Vm_t(\psi)
&=\alpha_{t}\Delta \overline\A_t+\vp_{t}\Delta \St_t\\
&=\frac{\pi_t\Vt_{t-1}(\psi)}{\St_{t-1}}\Delta \St_t\\
&=\Vt_{t-1}(\psi)\pi_t\bigg[\frac{b-r}{1+r}\D\widehat\Zm_{(t,1)}+\frac{a-r}{1+r}\D\widehat\Zm_{(t,-1)}\bigg]\\
&=\Vt_{t-1}(\psi)\pi_t\bigg[\frac{b-r+\rho(a-r)}{1+r}\D\widehat\Rm_{(t,1)}+\frac{a-r}{1+r}\D\widehat\Rm_{(t,-1)}\bigg]
\end{align*}
\noindent
Then,
\begin{equation*}
\Vt_{T-1}(\psi)=\Vm_0(\psi)+\sum_{t\in\id{1}{T-1}}\Vt_{t-1}(\psi)\pi_t\bigg[\frac{b-r+\rho(a-r)}{1+r}\D\widehat\Rm_{(t,1)}+\frac{a-r}{1+r}\D\widehat\Rm_{(t,-1)}\bigg].
\end{equation*}
Recall that we have assumed $\Fm=\Vm_{T-1}(\psi)=(1+r)^{T-1}\Vt_{T-1}(\psi)$. The uniqueness of the Clark formula \eqref{Clark_wQformula_eq}, entails $\Vm_0(\psi)=\gE_{\wQ}[\Fm|\g{0}]/(1+r)^{T-1}$ and for all $t\in\id{1}{T-1}$,
\begin{align*}
\gE_{\wQ}\big[(a-r)^{-1}\Dm_{(t,-1)}\Fm\,|\,\g{t-1}\big]=\frac{(1+r)^{T-1}}{1+r}\,\Vt_ {t-1}(\psi)\pi_t=\frac{(1+r)^{T-1}}{(1+r)^t}\,\Vm_ {t-1}(\psi)\pi_t.
\end{align*}
On the one hand, we set  $\varphi_0=0$ and
\begin{equation*}
\vp_t=\frac{\Vm_{t-1}(\psi)\,\pi_t}{\Sm_{t-1}}=(1+r)^{t-T+1}\frac{\gE_{\wQ}\big[\Dm_{(t,-1)}\Fm\,|\,\g{t-1}\big]}{(a-r)\Sm_{t-1}}.
\end{equation*}
%
On the other hand, let $\alpha_0=(1+r)^{-T+1}\gE_{\wQ}[\Fm|\G_0]=\gE[\Fm]$ (since $\wQ$ coincides with $\gP$ on $\sigma(\Gm)$ from \textcolor{teal}{\textbf{Fact 3}}) and for any $t\in\id{1}{T-1}$,
\begin{equation*}
\alpha_t=\alpha_{t-1}-\frac{(\varphi_t-\varphi_{t-1})\Sm_{t-1}}{\A_{t-1}}.
\end{equation*}
Reciprocally, we can check that $(\alpha,\varphi)$ defines a $\G$-admissible strategy with terminal value $\Fm$.\vspace{3pt}\\
\underline{\textbf{Step 2: Free choice of the martingale measure}}
We can check that the value of the strategy does not depend on the specific \textit{optimal} $\G$-martingale measure $\wQ\in\sP^{\mathrm{opt},\G}$ chosen. Consider two elements $\wQ^{\Um_1}$ and $\wQ^{\Um_2}$ in $\sP^{\mathrm{opt},\G}$ (defined by \eqref{Eq_insider_opt_measure_set}) of the form
\begin{equation*}
\wQ^{\Um_i}=(1-\lambda)\gPc+\lambda(\Um_{i}*\wQb)\in\sP^{\mathrm{opt},\G}, \quad i\in\{1,2\}, 
\end{equation*}
where $\Um_1,\Um_2\in\mathscr U^{\Bi,\G}$. We have
\begin{equation*}
\gE_{\wQ^{\Um_1}}[\Fm|\g{0}]=\frac{\gE_{\wQ^{\Um_2}}[(\Um_1/\Um_2)\cdot\Fm|\g{0}]}{\gE_{\wQ^{\Um_2}}[(\Um_1/\Um_2)|\g{0}]}=\gE_{\wQ^{\Um_2}}[\Fm|\g{0}],
\end{equation*}
since $\Um_1$ and $\Um_2$ are $\g{0}$-measurable. Similarly, we can prove that $\gE_{\wQ^{\Um}}\big[\Dm_{(t,-1)}\Fm\,|\,\g{t-1}\big]$ does not depend on the choice of $\wQ^{\Um}\in\sP^{\mathrm{opt},\G}$.\\
Thus, we get a couple of $\G$-predictable processes $\psi=(\alpha,\varphi)$ that satisfies the self-financing condition, with terminal value $\Fm$ and whose definition does not depend of the \textit{optimal} martingale measure chosen. 
Hence the result.
\end{proof}
\end{document}